\documentclass[11pt]{article}
\usepackage{amssymb}
\usepackage{amsmath}
\usepackage[top=1.2in, bottom=1.4in, left=1in, right=1in]{geometry}
\usepackage{parskip}
 \usepackage[T1]{fontenc}
\usepackage{lmodern}
 \usepackage{centernot}
\usepackage{epsfig}
\usepackage{xcolor}
\usepackage{amsmath,amsthm,amsfonts,amssymb,cite,amscd}

\usepackage{setspace}

\usepackage{fullpage}

\usepackage{indentfirst}
\usepackage{mathrsfs}
\usepackage{amsthm}
\usepackage{tikz}
\usepackage{mathtools}
\usepackage{calc}
\usepackage[symbol]{footmisc}
\usetikzlibrary{patterns,arrows,decorations.pathreplacing}
\usepackage{authblk}
\usetikzlibrary{fadings}
\usepackage{pgfplots}
\usepackage{parskip}
\newtheorem{theorem}{Theorem} 
\newtheorem{lemma}[theorem]{Lemma}
 
\newtheorem{definition}{Definition}

\newtheorem{conjecture}{Conjecture}

\newcommand{\RNum}[1]{\uppercase\expandafter{\romannumeral #1\relax}}
\DeclareMathOperator{\ex}{ex}

\DeclareMathOperator{\ar}{ar}

\def\0{\emptyset}
\def\q{\hfill\rule{1ex}{1ex}}

\pgfplotsset{compat=1.16}
\usepackage{authblk}
\usetikzlibrary{patterns}

\title{ Tur\'an numbers and anti-Ramsey numbers for short cycles in complete $3$-partite graphs}
\author[1]{\hspace{1cm}Chunqiu Fang \thanks{chunqiu@ustc.edu.cn}}
\author[2,3]{Ervin Gy\H{o}ri \thanks{gyori.ervin@renyi.mta.hu}} 
\author[2,3]{Chuanqi Xiao \thanks{chuanqixm@gmail.com}}
\author[4]{Jimeng Xiao \thanks{xiaojimeng@mail.nwpu.edu.cn}}
\affil[1]{School of Mathematical Sciences, University of Science and Technology of China, Hefei}
\affil[2]{Alfr\'ed R\'enyi Institute of Mathematics, Budapest }
\affil[3]{Central European University, Budapest}
\affil[4]{School of Mathematics and Statistics, Northwestern Polytechnical University, Xi'an}
\date{}
\begin{document}
\maketitle
\begin{abstract}

We call a $4$-cycle in $K_{n_{1}, n_{2}, n_{3}}$ multipartite, denoted by $C_{4}^{\text{multi}}$, if it contains at least one vertex in each part of $K_{n_{1}, n_{2}, n_{3}}$. The Tur\'an number $\ex(K_{n_{1},n_{2},n_{3}}, C_{4}^{\text{multi}})$ \bigg( respectively,  $\ex(K_{n_{1},n_{2},n_{3}},\{C_{3}, C_{4}^{\text{multi}}\})$\bigg) is the maximum number of edges in a graph $G\subseteq K_{n_{1},n_{2},n_{3}}$ such that $G$ contains no $C_{4}^{\text{multi}}$ \bigg( respectively, $G$ contains neither $C_{3}$ nor $C_{4}^{\text{multi}}$\bigg). We call a $C^{multi}_4$ rainbow if all four edges of it have different colors. The anti-Ramsey number $\ar(K_{n_{1},n_{2},n_{3}}, C_{4}^{\text{multi}})$ is the maximum number of colors in an edge-colored of $K_{n_{1},n_{2},n_{3}}$ with no rainbow  $C_{4}^{\text{multi}}$. In this paper, we determine that $\ex(K_{n_{1},n_{2},n_{3}}, C_{4}^{\text{multi}})=n_{1}n_{2}+2n_{3}$ and
$\ar(K_{n_{1},n_{2},n_{3}}, C_{4}^{\text{multi}})=\ex(K_{n_{1},n_{2},n_{3}}, \{C_{3}, C_{4}^{\text{multi}}\})+1=n_{1}n_{2}+n_{3}+1,$ where $n_{1}\ge  n_{2}\ge n_{3}\ge 1.$
\end{abstract}
\baselineskip=0.3in
\section{Introduction}

We consider only nonempty simple graphs. Let $G$ be such a graph, the vertex and edge set of $G$ is denoted by $V(G)$ and $E(G)$, the number of vertices and edges in $G$ by $\nu(G)$ and $e(G)$, respectively. We denote the neighborhood of $v$ in $G$ by $N_G(v)$, and the degree of a vertex $v$ in $G$ by $d_G(v)$, the size of $N_G(v)$. Let $U_1$, $U_2$ be vertex sets, denote by $e_{G}(U_1,U_2)$ the number of edges between $U_1$ and $U_2$ in $G$. We write $d(v)$ instead of $d_{G}(v)$, $N(v)$ instead of $N_{G}(v)$ and $e(U_1,U_2)$ instead of $e_G(U_1,U_2)$ if the underlying graph $G$ is clear. 

Given a graph family $\mathcal{F}$, we call a graph $H$ an $\mathcal{F}$-free graph, if $G$ contains no graph in $\mathcal{F}$ as a subgraph. The Tur\'an number $\ex(G, \mathcal{F})$ for a graph family $\mathcal{F}$ in $G$ is the maximum number of edges in a graph $H\subseteq G$ which is $\mathcal{F}$-free. If $\mathcal{F}=\{F\}$, then we denote $\ex(G,\mathcal{F})$ by $\ex(G,F)$.

An old result of Bollob\'as, Erd\H os and Szemer\'edi \cite{Bollobas1} showed that $\ex(K_{n_{1},n_{2},n_{3}},C_{3})=n_{1}n_{2}+n_{1}n_{3}$ for $n_{1}\ge n_{2}\ge n_{3}\ge 1$ (also see \cite{Bollobas2,Bennett,De}). Lv, Lu and Fang \cite{Lv1,Lv2} constructed balanced $3$-partite graphs which are $C_{4}$-free and $\{C_{3},C_{4}\}$-free respectively and showed that $\ex(K_{n,n,n}, C_{4})=(\frac{3}{\sqrt{2}}+o(1))n^{3/2}$ and $\ex(K_{n,n,n},\{C_{3},C_{4}\})\ge (\sqrt{3}+o(1))n^{3/2}.$ Recently, Fang, Gy\H ori, Li and Xiao \cite{Fang} showed that if $G\subseteq K_{n_{1}, n_{2}, \ldots, n_{r}}$ and $e(G)\ge f(n_{1},n_{2},\ldots,n_{r})+1$, then $G$ contains a multipartite cycle. Further more, they proposed the following conjecture.
\begin{conjecture} \cite{Fang}\label{conjecture1}
For $r\ge 3$ and $n_{1}\ge n_{2}\ge \cdots\ge n_{r}\ge 1$, if $G\subset K_{n_{1}, n_{2}, \ldots, n_{r}}$ and $e(G)\ge f(n_{1},n_{2},\ldots,n_{r})+1$, then $G$ contains a multipartite cycle of length no more than $\frac{3}{2}r$.
\end{conjecture}

The definitions of the multipartite subgraphs and the function $f(n_{1},n_{2},\ldots,n_{r})$ are defined as follows.

\begin{definition} \cite{Fang}
Let $r\ge 3$ and $G$ be an $r$-partite graph with vertex partition $V_1,V_2,\dots,V_r$, we call a subgraph $H$ of $G$ {\em multipartite}, if there are at least three distinct parts $V_{i}, V_{j}, V_{k}$ such that  $V(H)\cap V_{i}\ne \0, V(H)\cap V_{j}\ne \0$ and $V(H)\cap V_{k}\ne \0$. In particular, we denote a multipartite $H$ by  $H^{\text{multi}}$ (see Figure \ref{exC4} for an example of a $C_4^{multi}$ in a $3$-partite graph).
\end{definition}

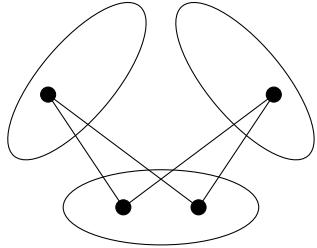
\begin{figure}[h]\label{exc4}
\begin{center}
\begin{tikzpicture}

\draw[rotate around={40:(-17,5)}] (-17,5) arc(0:360: .5cm and 1.3cm);
\draw[rotate around={140:(-20,5)}] (-20,5) arc(0:360: .5cm and 1.3cm);
\draw (-17.2,3) arc(0:360: 1.3cm and 0.5cm);

\draw [fill](-17,4.5)circle [radius = 0.1];
\draw [fill](-20,4.5)circle [radius = 0.1];
\draw [fill](-18,3)circle [radius = 0.1];
\draw [fill](-19,3)circle [radius = 0.1];

\draw  (-17,4.5) -- (-18,3);
\draw  (-17,4.5) -- (-19,3);
\draw  (-20,4.5) -- (-18,3);
\draw  (-20,4.5) -- (-19,3);

\end{tikzpicture}
\caption{A $C_4^{multi}$ in a $3$-partite graph.}
\end{center}
\end{figure}

For $r\ge 3$ and $n_{1}\ge n_{2}\ge \cdots\ge n_{r}\ge 1$, let
\begin{align}\nonumber
f(n_{1},n_{2},\ldots,n_{r})=
 \begin{cases}
 n_{1}n_{2}+n_{3}n_{4}+ \cdots+n_{r-2}n_{r-1}+n_{r}+\frac{r-1}{2}-1,& r\, \text{ is odd};\\
n_{1}n_{2}+n_{3}n_{4}+ \cdots+n_{r-1}n_{r}+\frac{r}{2}-1,& r \,\text{ is even}.
 \end{cases}
 \end{align}
 
 In this paper, we consider the Tur\'an numbers of $C_{4}^{\text{multi}}$ and $\{C_{3}, C_{4}^{\text{multi}}\}$ in the complete $3$-partite graphs and obtain the following two results.

\begin{theorem}\label{exC4}
For $n_{1}\ge n_{2}\ge n_{3}\ge 1$, $\ex(K_{n_{1},n_{2},n_{3}}, C_{4}^{\text{multi}})=n_{1}n_{2}+2n_{3}.$
\end{theorem}

\begin{theorem}\label{exC3C4}
For $n_{1}\ge n_{2}\ge n_{3}\ge 1$, $\ex(K_{n_{1},n_{2},n_{3}}, \{C_{3}, C_{4}^{\text{multi}}\})=n_{1}n_{2}+n_{3}.$
\end{theorem}

Notice that Theorem \ref{exC3C4} confirms Conjecture \ref{conjecture1} for the case when $r=3$. 

A subgraph of an edge-colored graph is \textit{rainbow}, if all of its edges have different colors. For graphs $G$ and $H$, the anti-Ramsey number $\ar(G,H)$ is the maximum number of colors in an edge-colored $G$ with no rainbow copy of $H$. Erd\H os, Simonovits and S\'{o}s \cite{Erdos}  first studied the anti-Ramsey number in the case when the host graph $G$ is a complete graph $K_{n}$ and showed the close relationship between it and the Tur\'an number. In this paper, we consider the anti-Ramsey number of $C_{4}^{\text{multi}}$ in the complete $3$-partite graphs. 


\begin{theorem}\label{arC4}
For $n_{1}\ge n_{2}\ge n_{3}\ge 1$, $\ar(K_{n_{1},n_{2},n_{3}}, C_{4}^{\text{multi}})=n_{1}n_{2}+n_{3}+1$.

\end{theorem}

We prove Theorems \ref{exC4} and \ref{exC3C4} in Section 2 and Theorem \ref{arC4} in Section 3, respectively. We always denote the vertex partition of $K_{n_{1},n_{2},n_{3}}$ by $V_{1}, V_{2}$ and $V_{3}$, where $|V_{i}|=n_{i}$, $1\le i\le 3$.

\section{The Tur\'an numbers of $C_{4}^{\text{multi}}$ and $\{C_{3},C_{4}^{\text{multi}}\}$}

In this section,we first give the following lemma which will play an important role in our proof.

\begin{lemma} \label{lemma1}

Let $G$ be a $3$-partite graph with vertex partition $X, Y$ and $Z$, such that for all  $x\in X$, $N(x)\cap Y\ne \0$ and $N(x)\cap Z\ne \0$.

(i) If $G$ is $C^{\text{multi}}_{4}$-free, then $e(G)\le |Y||Z|+2|X|$;

(ii) If $G$ is $\{C_{3}, C^{\text{multi}}_{4}\}$-free, then $e(G)\le |Y||Z|+|X|$.

\end{lemma}

\begin{proof}
(i) Since $G$ is $C_{4}^{\text{multi}}$-free, $G[N(x)]$ is $K_{1,2}$-free for each $x\in X$. Therefore,
\begin{align}\label{2.1}
e(G[N(x)])=e\bigg(N(x)\cap Y, N(x)\cap Z\bigg)\leq \min\bigg\{|N(x)\cap Y|, |N(x)\cap Z|\bigg\}.
\end{align}

For $x\in X$, we let $e_{x}$ be the number of missing edges of $G$ between $N(x)\cap Y$ and $N(x)\cap Z$. By (\ref{2.1}), we have
\begin{align}\label{2.2}
e_{x}&=|N(x)\cap Y|\cdot|N(x)\cap Z|-e\bigg(N(x)\cap Y, N(x)\cap Z\bigg)\nonumber\\
&\geq|N(x)\cap Y|\cdot|N(x)\cap Z|-\min\bigg\{|N(x)\cap Y|, |N(x)\cap Z|\bigg\}\\
&\geq|N(x)\cap Y|+|N(x)\cap Z|-2,\nonumber
\end{align}
where the last inequality holds since $|N(x)\cap Y|\geq 1$ and $|N(x)\cap Z|\geq 1$ for all $x\in X$.

By (\ref{2.2}), we get
\begin{align}\label{2.3}
 \sum_{x\in X}e_{x}\geq \sum_{x\in X}\bigg(|N(x)\cap Y|+|N(x)\cap Z|-2\bigg)=e(X,Y)+e(X, Z)-2|X|.
 \end{align}
 Notice that for any two distinct vertices $x_{1}, x_{2}\in X$,
they can not have common neighboors in both $Y$ and $Z$ at the same time, 
otherwise we find a copy of $C^{multi}_4$ in $G$. Thus each missing edge between $Y$ and $Z$ be calculated at most once in the summation $\sum_{x\in X}e_{x}$. Hence the number of missing edges between $Y$ and $Z$ is at least $\sum_{x\in X}e_{x}$. Then we have
\begin{align}\label{2.4}
e(Y,Z)\le |Y||Z|-\sum_{x\in X}e_{x}\le |Y||Z| - (e(X,Y)+e(X, Z)-2|X|).
\end{align}
By (\ref{2.4}), we get
\[\begin{split}
e(G)&=e(X,Y)+e(X, Z)+e(Y, Z)\le |Y||Z|+2|X|.
\end{split}\]

(ii) Since $G$ is $C_{3}$-free, for each $x\in X$,
\begin{align}\label{2.5}
e\bigg(N(x)\cap Y, N(x)\cap Z\bigg)=0.
\end{align}
 Since for each $x\in X$, $|N(x)\cap Y|\geq 1$ and $|N(x)\cap Z|\geq 1$ hold, by (\ref{2.5}), the number of missing edges between $N(x)\cap Y$ and $N(x)\cap Z$ is $|N(x)\cap Y|\cdot|N(x)\cap Z|$.
Notice that for any two distinct vertices $x_{1}, x_{2}\in X$, they cannot have common neighboors in both $Y$ and $Z$ at the same time, otherwise we find a copy of $C^{multi}_4$ in $G$. Hence,  the number of missing edges between $Y$ and $Z$ is at least $\sum_{x\in X}|N(x)\cap Y|\cdot|N(x)\cap Z|$. Thus,
\begin{align}\label{2.6}
e(Y, Z)&\le |Y||Z|-\sum_{x\in X}|N(x)\cap Y|\cdot|N(x)\cap Z|\nonumber\\
&\le |Y||Z|-\sum_{x\in X}(|N(x)\cap Y|+|N(x)\cap Z|-1)\\
&=|Y||Z|+|X|-e(X, Y)-e(X,Z),\nonumber
\end{align}
the second inequality holds since $|N(x)\cap Y|\geq 1$ and $|N(x)\cap Z|\geq 1$ for $x\in X$.

By (\ref{2.6}), we have $e(G)=e(Y, Z)+e(X, Y)+e(X,Z)\le |Y||Z|+|X|.$
\end{proof}

Now we are able to  prove Theorems \ref{exC4} and \ref{exC3C4}.

\begin{proof}[Proof of Theorem \ref{exC4}]
Let $G\subseteq K_{n_1,n_2,n_3}$ be a graph, such that $V_1$ and $V_2$ are completely joined, $V_1$ (respectively, $V_2$) and $V_3$ are joined by an $n_3$-matching, see Figure \ref{l1}. Clearly, $G$ is $C^{multi}_4$-free and $e(G)=n_1n_2+2n_3$. Therefore, $\ex(K_{n_1,n_2,n_3}, C^{multi}_4)\geq n_1n_2+2n_3$.

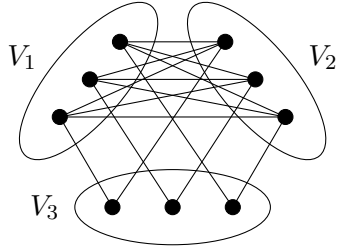
\begin{figure}[h]
\begin{center}
\begin{minipage}{5cm}
\begin{tikzpicture}
\draw[rotate around={40:(-12,-5)}] (-12,-5) arc(0:360: .5cm and 1.3cm);
\draw[rotate around={140:(-15,-5)}] (-15,-5) arc(0:360: .5cm and 1.3cm);
\draw (-12.2,-7) arc(0:360: 1.3cm and 0.5cm);

\draw [fill](-12,-5.8)circle [radius = 0.1];
\draw [fill](-12.4,-5.3)circle [radius = 0.1];
\draw [fill](-12.8,-4.8)circle [radius = 0.1];
\draw [fill](-15,-5.8)circle [radius = 0.1];
\draw [fill](-14.6,-5.3)circle [radius = 0.1];
\draw [fill](-14.2,-4.8)circle [radius = 0.1];
\draw [fill](-12.7,-7)circle [radius = 0.1];
\draw [fill](-13.5,-7)circle [radius = 0.1];
\draw [fill](-14.3,-7)circle [radius = 0.1];

\node at(-15.5, -5){$V_1$};
\node at(-11.5, -5){$V_2$};
\node at(-15.2, -7){$V_3$};

\draw  (-12.8,-4.8) -- (-14.2,-4.8);
\draw  (-12.8,-4.8) -- (-15,-5.8);
\draw  (-12.8,-4.8) -- (-14.6,-5.3);

\draw  (-12,-5.8) -- (-15,-5.8);
\draw  (-12,-5.8) -- (-14.2,-4.8);
\draw  (-12,-5.8) -- (-14.6,-5.3);

\draw  (-12.4,-5.3) -- (-14.2,-4.8);
\draw  (-12.4,-5.3) -- (-15,-5.8);
\draw  (-12.4,-5.3) -- (-14.6,-5.3);

\draw  (-14.3,-7) -- (-12.8,-4.8);
\draw  (-14.3,-7) -- (-15,-5.8);
\draw  (-13.5,-7) -- (-12.4,-5.3);
\draw  (-13.5,-7) -- (-14.6,-5.3);
\draw  (-12.7,-7) -- (-14.2,-4.8);
\draw  (-12.7,-7) -- (-12,-5.8);

\end{tikzpicture}
\end{minipage}

\caption{An example of $C^{multi}_4$-free graph with $n_1n_2+2n_3$ edges.}
\end{center}\label{l1}
\end{figure}

Let $G\subseteq K_{n_{1},n_{2},n_{3}}$ such that $G$ is $C^{\text{multi}}_{4}$-free, now we are going to prove that $e(G)\le n_{1}n_{2}+2n_{3}$ by induction on $n_{1}+n_{2}+n_{3}$. 

For the base case $n_{3}=1$, let $V_{3}=\{v\}$, we consider the following four subcases: \\
$(i)$ $N(v)\cap V_{1}\ne \0$ and $N(v)\cap V_{2}\ne \0$, then by Lemma \ref{lemma1}, we have $e(G)\le n_{1}n_{2}+2$. \\
$(ii)$ $N(v)\cap V_{1}\ne \0$ and $N_{G}(v)\cap V_{2}= \0$,  then
\[\begin{split}e(G)&=e(V_{3}, N(v))+e(V_{2}, N(v))+e(V_{1}\setminus N(v), V_{2})\\
&\le d(v)+n_{2}+\bigg(n_{1}-d(v)\bigg)n_{2}\\
&\le n_{1}n_{2}+1.
\end{split}\]
$(iii)$ $N(v)\cap V_{1}= \0$ and $N(v)\cap V_{2}\ne \0$,  then
\[\begin{split}e(G)&=e(V_{3}, N(v))+e(V_{1}, N(v))+e(V_{2}\setminus N(v), V_{1})\\
&\le d(v)+n_{1}+(n_{2}-d(v))n_{1}\\
&\le n_{1}n_{2}+1.
\end{split}\]
$(iv)$ $N(v)\cap V_{1}= \0$ and $N(v)\cap V_{2}= \0$, then $e(G)= e(V_{1}, V_{2})\le n_{1}n_{2}.$

Now let $n_{3}\ge 2$, and assume that the conclusion is true for order less than $n_{1}+n_{2}+n_{3}.$ We consider the following three cases.

{\bf Case 1.} $n_{1}=n_{2}=n_{3}=n\ge 2.$

If there exists one part, say $V_{1}$, such that $N(v)\cap V_{2}\ne\0$ and $N(v)\cap V_{3}\ne\0$, for all $v\in V_{1}$, then by Lemma \ref{lemma1}, we have $e(G)\le |V_{2}||V_{3}|+2|V_{1}|=n^{2}+2n.$

Thus, we may assume that for all $i\in [3]$, there exist a vertex $v\in V_{i}$ and $j\in [3]\setminus\{i\}$ such that $N(v)\cap V_{j}=\0.$ We separate it into two subcases.

{\bf Case 1.1.} There exist two parts, say $V_{1}$ and $V_{2}$, such that $N(v_{1})\cap V_{2}=\0$ and $N(v_{2})\cap V_{1}=\0$ for some vertices $v_{1}\in V_{1}$ and $v_{2}\in V_{2}.$

Since $G$ is $C_{4}^{\text{multi}}$-free, $d(v_{1})+d(v_{2})\le |V_{3}|+1=n+1.$ Without loss of generality, let $v_{3}\in V_{3}$ be the vertex such that $N(v_{3})\cap V_{1}=\0$. Then the number of edges incident with $\{v_{1}, v_{2}, v_{3}\}$ in $G$ is at most $d(v_{1})+d(v_{2})+n-1\le 2n.$ By the induction hypothesis, $e(G-\{v_{1}, v_{2}, v_{3}\})\le (n-1)^{2}+2(n-1)$. Thus, $e(G)\le (n-1)^{2}+2(n-1)+2n\le n^{2}+2n.$

{\bf Case 1.2.} There exist vertices $v_{1}\in V_{1}, v_{2}\in V_{2}$ and $v_{3}\in V_{3}$ such that either $N(v_{1})\cap V_{2}=\0, N(v_{2})\cap V_{3}=\0, N(v_{3})\cap V_{1}=\0$ or $N(v_{1})\cap V_{3}=\0, N(v_{3})\cap V_{2}=\0, N(v_{2})\cap V_{1}=\0$ holds.

Without loss of generality, we assume that $N(v_{1})\cap V_{2}=\0, N(v_{2})\cap V_{3}=\0, N(v_{3})\cap V_{1}=\0$. If $d(v_{1})+d(v_{2})+d(v_{3})\le 2n+1$, then by the induction hypothesis, we have
\[\begin{split}
e(G)&\le e(G-\{v_{1}, v_{2}, v_{3}\})+d(v_{1})+d(v_{2})+d(v_{3})\\
&\le (n-1)^{2}+2(n-1)+2n+1\\
&\le n^{2}+2n.
\end{split}\]
Now we assume that $d(v_{1})+d(v_{2})+d(v_{3})\ge 2n+2$, hence, $d(v_{1})\ge 1, d(v_{2})\ge 1, d(v_{3})\ge 1$. Since $G$ is $C^{multi}_4$-free, each vertex in $V_1\setminus \{v_1\}$ can have at most one neighbour in $N(v_3)$, we have $e(V_1\setminus \{v_1\},N(v_3))\leq n-1$. Similarly, we have $e(V_{3}\setminus\{v_{3}\}, N(v_{2}))\leq n-1$ and $e(V_{2}\setminus\{v_{2}\}, N(v_{1}))\leq n-1$.

Therefore,
$$e(V_{1}, V_{2})= e(V_{1}\setminus\{v_{1}\}, V_{2}\setminus N(v_{3}))+e(V_{1}\setminus\{v_{1}\}, N(v_{3}))\le (n-d(v_{3}))(n-1)+(n-1),$$
$$e(V_{1}, V_{3})= e(V_{3}\setminus\{v_{3}\}, V_{1}\setminus N(v_{2}))+e(V_{3}\setminus\{v_{3}\}, N(v_{2}))\le (n-d(v_{2}))(n-1)+(n-1),$$
$$e(V_{2}, V_{3})= e(V_{2}\setminus\{v_{2}\}, V_{3}\setminus N(v_{1}))+e(V_{2}\setminus\{v_{2}\}, N(v_{1}))\le (n-d(v_{1}))(n-1)+(n-1).$$

Thus,
\[\begin{split}
e(G)&=e(V_{1}, V_{2})+e(V_{1}, V_{3})+e(V_{2}, V_{3})\\
&\le \bigg(3n-(d(v_{1})+d(v_{2})+d(v_{3}))\bigg)(n-1)+3(n-1)\\
&\le\bigg(3n-(2n+2)\bigg)(n-1)+3(n-1)\\
&\le n^{2}-1.
\end{split}\]

{\bf Case 2.} $n_{1}>n_{2}=n_{3}=n\ge 2.$

If there exists one vertex $v_{0}\in V_{1}$ such that $d(v_{0})\le n$, by the induction hypothesis, we have $e(G)= e(G-v_{0})+d(v_{0})\le (n_{1}-1)n+2n+n\le n_{1}n+2n.$ Otherwise, we have $d(v)\ge n+1$ for all vertices $v\in V_{1}$. Hence,  $N(v)\cap V_{2}\ne \0$ and $N(v)\cap V_{3}\ne \0$ hold for all $v\in V_{1}$. By Lemma \ref{lemma1}, we get $e(G)\le n^{2}+2n_{1}\le n_{1}n+2n.$

{\bf Case 3.} $n_{1}\ge n_{2}>n_{3}\ge 2.$

If there exists one vertex $v_{0}\in V_{2}$ such that $d(v_{0})\le n_{1}$, by the induction hypothesis, we have $e(G)= e(G-v_{0})+d(v_{0})\le n_{1}(n_{2}-1)+2n_{3}+n_{1}\le n_{1}n_{2}+2n_{3}.$ Otherwise, we have $d(v)\ge n_{1}+1$ for all vertices $v\in V_{2}$. Hence, $N(v)\cap V_{1}\ne \0$ and $N(v)\cap V_{3}\ne \0$ for all $v\in V_{2}$. By Lemma \ref{lemma1}, we get $e(G)\le n_{1}n_{3}+2n_{2}\le n_{1}n_{2}+2n_{3}.$
\end{proof}

\begin{proof} [Proof of Theorem \ref{exC3C4}]
Let $G\subseteq K_{n_1,n_2,n_3}$ be a graph, such that $V_1$ and $V_2$ are completely joined, $V_1$  and $V_3$ are joined by an $n_3$-matching and there is no edge between $V_2$ and $V_3$, see Figure \ref{l2}. Clearly, $G$ is $\{C_3,C^{multi}_4\}$-free and $e(G)=n_1n_2+n_3$. Therefore, $\ex(K_{n_1,n_2,n_3}, \{C_3,C^{multi}_4\})\geq n_1n_2+n_3$.

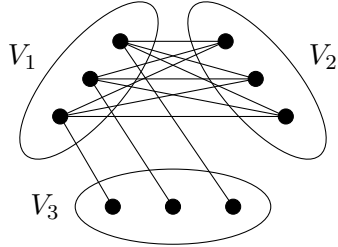
\begin{figure}[h]
\centering
\begin{tikzpicture}
\draw[rotate around={40:(-12,-5)}] (-12,-5) arc(0:360: .5cm and 1.3cm);
\draw[rotate around={140:(-15,-5)}] (-15,-5) arc(0:360: .5cm and 1.3cm);
\draw (-12.2,-7) arc(0:360: 1.3cm and 0.5cm);

\node at(-15.5, -5){$V_1$};
\node at(-11.5, -5){$V_2$};
\node at(-15.2, -7){$V_3$};

\draw [fill](-12,-5.8)circle [radius = 0.1];
\draw [fill](-12.4,-5.3)circle [radius = 0.1];
\draw [fill](-12.8,-4.8)circle [radius = 0.1];
\draw [fill](-15,-5.8)circle [radius = 0.1];
\draw [fill](-14.6,-5.3)circle [radius = 0.1];
\draw [fill](-14.2,-4.8)circle [radius = 0.1];
\draw [fill](-12.7,-7)circle [radius = 0.1];
\draw [fill](-13.5,-7)circle [radius = 0.1];
\draw [fill](-14.3,-7)circle [radius = 0.1];

\draw  (-12.8,-4.8) -- (-14.2,-4.8);
\draw  (-12.8,-4.8) -- (-15,-5.8);
\draw  (-12.8,-4.8) -- (-14.6,-5.3);

\draw  (-12,-5.8) -- (-15,-5.8);
\draw  (-12,-5.8) -- (-14.2,-4.8);
\draw  (-12,-5.8) -- (-14.6,-5.3);

\draw  (-12.4,-5.3) -- (-14.2,-4.8);
\draw  (-12.4,-5.3) -- (-15,-5.8);
\draw  (-12.4,-5.3) -- (-14.6,-5.3);

\draw  (-14.3,-7) -- (-15,-5.8);
\draw  (-13.5,-7) -- (-14.6,-5.3);
\draw  (-12.7,-7) -- (-14.2,-4.8);

\end{tikzpicture}
\caption{An example of $\{C_3, C^{multi}_4\}$-free graph with $n_1n_2+n_3$ edges.}
\label{l2}
\end{figure}

Let $G\subseteq K_{n_{1},n_{2},n_{3}}$ such that $G$ is $\{C_3,C^{\text{multi}}_{4}\}$-free, now we are going to prove $e(G)\le n_{1}n_{2}+n_{3}$ by induction on $n_{1}+n_{2}+n_{3}$. For the base case $n_{3}=1$, let $V_{3}=\{v\}$. We consider the following four subcases;

$(i)$ $N(v)\cap V_{1}\ne \0$ and $N(v)\cap V_{2}\ne \0$, then by Lemma \ref{lemma1}, $e(G)\le n_{1}n_{2}+1$.

$(ii)$  $N(v)\cap V_{1}\ne \0$ and $N(v)\cap V_{2}= \0$,  then
\[\begin{split} e(G)&=e(V_{3}, N(v))+e(V_{2}, N(v))+e(V_{2},V_{1}\setminus N(v) )\\
&\le d(v)+n_{2}+(n_{1}-d(v))n_{2}\\
&\le n_{1}n_{2}+1.
\end{split}\]

$(iii)$ $N(v)\cap V_{1}= \0$ and $N(v)\cap V_{2}\ne \0$,  then
\[\begin{split}e(G)&=e(V_{3}, N(v))+e(V_{1}, N(v))+e(V_{1},V_{2}\setminus N(v))\\
&\le d(v)+n_{1}+(n_{2}-d(v))n_{1}\\
&\le n_{1}n_{2}+1.
\end{split}\]

$(iv)$ $N(v)\cap V_{1}= \0$ and $N(v)\cap V_{2}= \0$, then $e(G)= e(V_{1}, V_{2})\le n_{1}n_{2}.$

Now let $n_{3}\ge 2$, and assume that the conclusion is correct for order less than $n_{1}+n_{2}+n_{3}.$ We consider the following three cases.

{\bf Case 1.} $n_{1}=n_{2}=n_{3}=n\ge 2.$

If there exists one part, say $V_{1}$, such that $N(v)\cap V_{2}\ne\0$ and $N(v)\cap V_{3}\ne\0$, for all $v\in V_{1}$, then by Lemma \ref{lemma1}, we have $e(G)\le |V_{2}||V_{3}|+|V_{1}|=n^{2}+n.$

Thus, we may assume that for all $i\in [3]$, there exists a vertex $v\in V_{i}$ and $j\in [3]\setminus\{i\}$ such that $N(v)\cap V_{j}=\0.$ We separate the proof into two subcases.

{\bf Case 1.1.} There exist two parts, say $V_{1}$ and $V_{2}$, such that $N(v_{1})\cap V_{2}=\0$ and $N(v_{2})\cap V_{1}=\0$ for some vertices $v_{1}\in V_{1}$ and $v_{2}\in V_{2}.$

Since $G$ is $\{C_{3}, C_{4}^{\text{multi}}\}$-free, $d(v_{1})+d(v_{2})\le |V_{3}|+1=n+1.$ Without loss of generality, let $v_{3}\in V_{3}$ be the vertex such that $N(v_{3})\cap V_{1}=\0$. Then the number of edges incident to $\{v_{1}, v_{2}, v_{3}\}$ in $G$ is at most $d(v_{1})+d(v_{2})+n-1\le 2n.$ By the induction hypothesis, $e(G-\{v_{1}, v_{2}, v_{3}\})\le (n-1)^{2}+(n-1)$. Thus, $e(G)\le (n-1)^{2}+(n-1)+2n\le n^{2}+n.$

{\bf Case 1.2.} There exist vertices $v_{1}\in V_{1}, v_{2}\in V_{2}$ and $v_{3}\in V_{3}$ such that either $N(v_{1})\cap V_{2}=\0, N(v_{2})\cap V_{3}=\0, N(v_{3})\cap V_{1}=\0$ or $N(v_{1})\cap V_{3}=\0, N(v_{3})\cap V_{2}=\0, N(v_{2})\cap V_{1}=\0$ holds.

Without loss of generality, we assume that $N(v_{1})\cap V_{2}=\0, N(v_{2})\cap V_{3}=\0, N(v_{3})\cap V_{1}=\0$. If $d(v_{1})+d(v_{2})+d(v_{3})\le 2n$, by the induction hypothesis, we have
\[\begin{split}
e(G)&\le e(G-\{v_{1}, v_{2}, v_{3}\})+d(v_{1})+d(v_{2})+d(v_{3})\\
&\le (n-1)^{2}+(n-1)+2n\\
&\le n^{2}+n.
\end{split}\]
Otherwise,  $d(v_{1})+d(v_{2})+d(v_{3})\ge 2n+1$, hence, $d(v_{1})\ge 1, d(v_{2})\ge 1$, and $d(v_{3})\ge 1$. Since $G$ is $\{C_{3}, C_{4}^{\text{multi}}\}$-free,
$$e(V_{1}, V_{2})= e(V_{1}\setminus\{v_{1}\}, V_{2}\setminus N(v_{3}))+e(V_{1}\setminus\{v_{1}\}, N(v_{3}))\le (n-d(v_{3}))(n-1)+(n-1),$$
$$e(V_{1}, V_{3})= e(V_{3}\setminus\{v_{3}\}, V_{1}\setminus N(v_{2}))+e(V_{3}\setminus\{v_{3}\}, N(v_{2}))\le (n-d(v_{2}))(n-1)+(n-1),$$
$$e(V_{2}, V_{3})= e(V_{2}\setminus\{v_{2}\}, V_{3}\setminus N(v_{1}))+e(V_{2}\setminus\{v_{2}\}, N(v_{1}))\le (n-d(v_{1}))(n-1)+(n-1).$$

Thus,
\[\begin{split}
e(G)&=e(V_{1}, V_{2})+e(V_{1}, V_{3})+e(V_{2}, V_{3})\\
&\le \bigg(3n-(d(v_{1})+d(v_{2})+d(v_{3}))\bigg)(n-1)+3(n-1)\\
&\le(3n-(2n+2))(n-1)+3(n-1)\\
&\le n^{2}-1.
\end{split}\]

{\bf Case 2.} $n_{1}>n_{2}=n_{3}=n\ge 2.$

If there exists one vertex $v_{0}\in V_{1}$ such that $d(v_{0})\le n$, by the induction hypothesis, we have $e(G)= e(G-v_{0})+d(v_{0})\le (n_{1}-1)n+n+n\le n_{1}n+n.$ Otherwise, $d(v)\ge n+1$ for all vertex $v\in V_{1}$. Thus $N(v)\cap V_{2}\ne \0$ and $N(v)\cap V_{3}\ne \0.$ By Lemma \ref{lemma1}, we have $e(G)\le n^{2}+n_{1}\le n_{1}n+n.$

{\bf Case 3.} $n_{1}\ge n_{2}>n_{3}\ge 2.$

If there exists one vertex $v_{0}\in V_{2}$ such that $d(v_{0})\le n_{1}$,  by the induction hypothesis, we have $e(G)= e(G-v_{0})+d(v_{0})\le n_{1}(n_{2}-1)+n_{3}+n_{1}\le n_{1}n_{2}+n_{3}.$ Otherwise, $d(v)\ge n_{1}+1$ for all vertex $v\in V_{2}$. Thus $N(v)\cap V_{1}\ne \0$ and $N(v)\cap V_{3}\ne \0.$ By Lemma \ref{lemma1}, we have $e(G)\le n_{1}n_{3}+n_{2}\le n_{1}n_{2}+n_{3}.$
\end{proof}

\section{The anti-Ramsey number of $C_{4}^{\text{multi}}$}

In this section, we study the anti-Ramsey number of $C_{4}^{\text{multi}}$ in the complete $3$-partite graphs. Given an edge-coloring $c$ of $G$, we denote the color of an edge $e$ by $c(e)$. For a subgraph $H$ of $G$, we denote $C(H) = \{c(e)|e \in E(H)\}$. We call a spanning subgraph of an edge-colored graph representing subgraph, if it contains exactly one edge of each color.

Given graphs $G_{1}$ and $G_{2}$, we use $G_{1}\wedge G_{2}$ to denote the graph consisting of $G_{1}$ and $G_{2}$ which intersect in exactly one common vertex. We call a multipartite $C_{6}$ in a $3$-partite graph non-cyclic if there exists a vertex $v$ in $C_6$ such that the two neigborhoods in $C_6$ of $v$  belong to the same part. Let $\mathcal{F}$ be a graph family which consists of $C_{4}^{\text{multi}}$ (see graph $G_1$ in Figure \ref{F}), $C_{3}\wedge C_{3}$ (see graph $G_2$ in Figure \ref{F}), the non-cyclic $C_{6}^{multi}$ (see graphs $G_3, G_4$ in Figure \ref{F}) and $C_{3}\wedge C_{5}$ (see graphs $G_5, G_6, G_7$ in Figure \ref{F}) 
and the $C_{8}^{\text{multi}}$ which contains at least two vertex-disjoint non-multipartite $P_{3}$ (see graph $G_8$ in Figure \ref{F}).

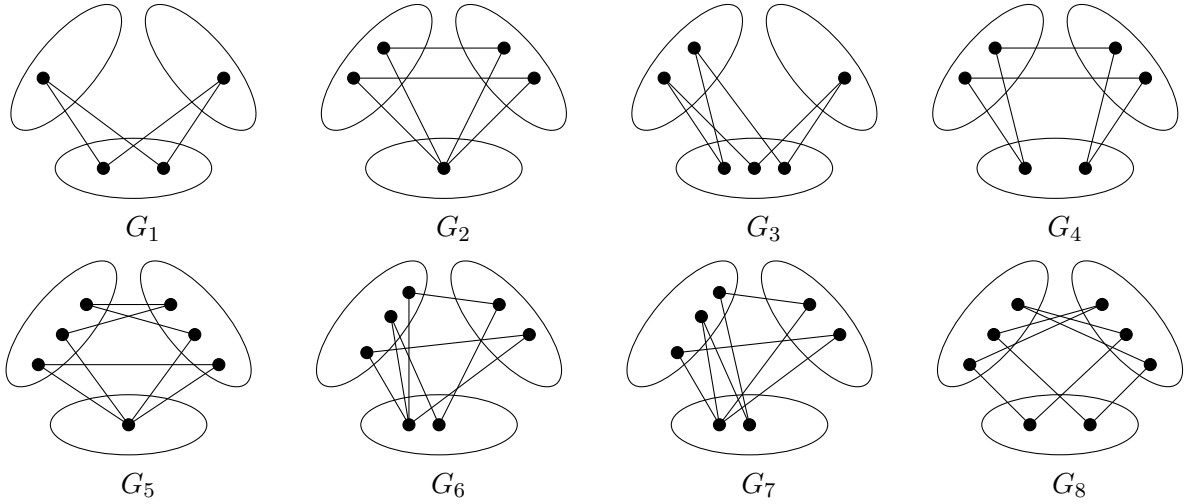
\begin{figure}[h]
\begin{center}
\begin{minipage}{4cm}
\begin{tikzpicture}[scale=0.8]

\draw[rotate around={40:(-17,5)}] (-17,5) arc(0:360: .5cm and 1.3cm);
\draw[rotate around={140:(-20,5)}] (-20,5) arc(0:360: .5cm and 1.3cm);
\draw (-17.2,3) arc(0:360: 1.3cm and 0.5cm);

\draw [fill](-17,4.5)circle [radius = 0.1];
\draw [fill](-20,4.5)circle [radius = 0.1];
\draw [fill](-18,3)circle [radius = 0.1];
\draw [fill](-19,3)circle [radius = 0.1];

\draw  (-17,4.5) -- (-18,3);
\draw  (-17,4.5) -- (-19,3);
\draw  (-20,4.5) -- (-18,3);
\draw  (-20,4.5) -- (-19,3);
\node[right] at(-18.8,2){$G_1$};
\end{tikzpicture}
\end{minipage}
\begin{minipage}{4cm}
\begin{tikzpicture}[scale=0.8]
\draw[rotate around={40:(-12,5)}] (-12,5) arc(0:360: .5cm and 1.3cm);
\draw[rotate around={140:(-15,5)}] (-15,5) arc(0:360: .5cm and 1.3cm);
\draw (-12.2,3) arc(0:360: 1.3cm and 0.5cm);

\draw [fill](-12,4.5)circle [radius = 0.1];
\draw [fill](-12.5,5)circle [radius = 0.1];
\draw [fill](-15,4.5)circle [radius = 0.1];
\draw [fill](-14.5,5)circle [radius = 0.1];
\draw [fill](-13.5,3)circle [radius = 0.1];

\draw  (-12,4.5) -- (-13.5,3);
\draw  (-12.5,5) -- (-13.5,3);
\draw  (-15,4.5) -- (-13.5,3);
\draw  (-14.5,5) -- (-13.5,3);
\draw  (-12,4.5) -- (-15,4.5);
\draw  (-12.5,5) -- (-14.5,5);
\node[right] at(-13.8,2){$G_2$};
\end{tikzpicture}
\end{minipage}
\begin{minipage}{8cm}
\begin{tikzpicture}[scale=0.8]
\draw[rotate around={40:(-17,0)}] (-17,0) arc(0:360: .5cm and 1.3cm);
\draw[rotate around={140:(-20,0)}] (-20,0) arc(0:360: .5cm and 1.3cm);
\draw (-17.2,-2) arc(0:360: 1.3cm and 0.5cm);

\draw [fill](-17,-0.5)circle [radius = 0.1];
\draw [fill](-20,-0.5)circle [radius = 0.1];
\draw [fill](-19.5,0)circle [radius = 0.1];
\draw [fill](-18,-2)circle [radius = 0.1];
\draw [fill](-18.5,-2)circle [radius = 0.1];
\draw [fill](-19,-2)circle [radius = 0.1];

\draw  (-20,-0.5) -- (-19,-2);
\draw  (-19.5,0) -- (-19,-2);
\draw  (-20,-0.5) -- (-18.5,-2);
\draw  (-18,-2) -- (-17,-0.5);
\draw  (-18.5,-2) -- (-17,-0.5);
\draw  (-18,-2) -- (-19.5,0);

\node[right] at(-18.8,-3){$G_3$};
\draw[rotate around={40:(-12,0)}] (-12,0) arc(0:360: .5cm and 1.3cm);
\draw[rotate around={140:(-15,0)}] (-15,0) arc(0:360: .5cm and 1.3cm);
\draw (-12.2,-2) arc(0:360: 1.3cm and 0.5cm);

\draw [fill](-12,-0.5)circle [radius = 0.1];
\draw [fill](-12.5,0)circle [radius = 0.1];
\draw [fill](-15,-0.5)circle [radius = 0.1];
\draw [fill](-14.5,0)circle [radius = 0.1];
\draw [fill](-13,-2)circle [radius = 0.1];
\draw [fill](-14,-2)circle [radius = 0.1];

\draw  (-12,-0.5) -- (-15,-0.5);
\draw  (-12.5,0) -- (-14.5,0);
\draw  (-12,-0.5) -- (-13,-2);
\draw  (-12.5,0) -- (-13,-2);
\draw  (-14.5,0) -- (-14,-2);
\draw  (-15,-0.5) -- (-14,-2);

\node[right] at(-13.8,-3){$G_4$};
\end{tikzpicture}
\end{minipage}\\ 
\begin{minipage}{4cm}
\begin{tikzpicture}[scale=0.8]
\draw[rotate around={40:(-7,5)}] (-7,5) arc(0:360: .5cm and 1.3cm);
\draw[rotate around={140:(-10,5)}] (-10,5) arc(0:360: .5cm and 1.3cm);
\draw (-7.2,3) arc(0:360: 1.3cm and 0.5cm);

\draw [fill](-7,4)circle [radius = 0.1];
\draw [fill](-7.8,5)circle [radius = 0.1];
\draw [fill](-7.4,4.5)circle [radius = 0.1];
\draw [fill](-10,4)circle [radius = 0.1];
\draw [fill](-9.2,5)circle [radius = 0.1];
\draw [fill](-9.6,4.5)circle [radius = 0.1];
\draw [fill](-8.5,3)circle [radius = 0.1];


\draw  (-7.4,4.5) -- (-8.5,3);
\draw  (-9.6,4.5) -- (-8.5,3);
\draw  (-7,4) -- (-8.5,3);
\draw  (-10,4) -- (-8.5,3);
\draw  (-7.4,4.5) -- (-9.2,5);
\draw  (-9.6,4.5) -- (-7.8,5);
\draw  (-7,4) -- (-10,4);
\draw  (-7.8,5) -- (-9.2,5);

\node[right] at(-8.8,2){$G_5$};
\end{tikzpicture}
\end{minipage}
\begin{minipage}{4cm}
\begin{tikzpicture}[scale=0.8]

\draw[rotate around={40:(-2,5)}] (-2,5) arc(0:360: .5cm and 1.3cm);
\draw[rotate around={140:(-5,5)}] (-5,5) arc(0:360: .5cm and 1.3cm);
\draw (-2.2,3) arc(0:360: 1.3cm and 0.5cm);

\draw [fill](-2,4.5)circle [radius = 0.1];
\draw [fill](-2.5,5)circle [radius = 0.1];
\draw [fill](-4,5.2)circle [radius = 0.1];
\draw [fill](-4.3,4.8)circle [radius = 0.1];
\draw [fill](-4.7,4.2)circle [radius = 0.1];
\draw [fill](-3.5,3)circle [radius = 0.1];
\draw [fill](-4,3)circle [radius = 0.1];


\draw  (-4,3) -- (-4,5.2);
\draw  (-2.5,5) -- (-3.5,3);
\draw (-2.5,5) -- (-4,5.2);
\draw  (-4,3) -- (-4.3,4.8);
\draw  (-4.7,4.2) -- (-2,4.5);
\draw  (-4,3) -- (-2,4.5);
\draw  (-4.7,4.2) -- (-4,3);
\draw  (-4.3,4.8) -- (-3.5,3);
\node[right] at(-3.8,2){$G_6$};
\end{tikzpicture}
\end{minipage}
\begin{minipage}{4cm}
\begin{tikzpicture}[scale=0.8]

\draw[rotate around={40:(-2,5)}] (-2,5) arc(0:360: .5cm and 1.3cm);
\draw[rotate around={140:(-5,5)}] (-5,5) arc(0:360: .5cm and 1.3cm);
\draw (-2.2,3) arc(0:360: 1.3cm and 0.5cm);

\draw [fill](-2,4.5)circle [radius = 0.1];
\draw [fill](-2.5,5)circle [radius = 0.1];
\draw [fill](-4,5.2)circle [radius = 0.1];
\draw [fill](-4.3,4.8)circle [radius = 0.1];
\draw [fill](-4.7,4.2)circle [radius = 0.1];
\draw [fill](-3.5,3)circle [radius = 0.1];
\draw [fill](-4,3)circle [radius = 0.1];

\draw  (-3.5,3) -- (-4,5.2);
\draw  (-2.5,5) -- (-4,3);
\draw (-2.5,5) -- (-4,5.2);

\draw  (-4,3) -- (-4.3,4.8);
\draw  (-4.7,4.2) -- (-2,4.5);
\draw  (-4,3) -- (-2,4.5);
\draw  (-4.7,4.2) -- (-4,3);
\draw  (-4.3,4.8) -- (-3.5,3);
\node[right] at(-3.8,2){$G_7$};
\end{tikzpicture}
\end{minipage}
\begin{minipage}{4cm}
\begin{tikzpicture}[scale=0.8]
\draw[rotate around={40:(-7,-5)}] (-7,-5) arc(0:360: .5cm and 1.3cm);
\draw[rotate around={140:(-10,-5)}] (-10,-5) arc(0:360: .5cm and 1.3cm);
\draw (-7.2,-7) arc(0:360: 1.3cm and 0.5cm);

\draw [fill](-7,-6)circle [radius = 0.1];
\draw [fill](-7.8,-5)circle [radius = 0.1];
\draw [fill](-7.4,-5.5)circle [radius = 0.1];
\draw [fill](-10,-6)circle [radius = 0.1];
\draw [fill](-9.2,-5)circle [radius = 0.1];
\draw [fill](-9.6,-5.5)circle [radius = 0.1];
\draw [fill](-8,-7)circle [radius = 0.1];
\draw [fill](-9,-7)circle [radius = 0.1];

\draw  (-7.8,-5) -- (-9.6,-5.5);
\draw  (-7.4,-5.5) -- (-9.2,-5);
\draw  (-7,-6) -- (-9.2,-5);
\draw  (-7.8,-5) -- (-10,-6);
\draw  (-8,-7) -- (-9.6,-5.5);
\draw  (-7.4,-5.5) -- (-9,-7);
\draw  (-7,-6) -- (-8,-7);
\draw  (-9,-7) -- (-10,-6);
\node[right] at(-8.8,-8){$G_8$};
\end{tikzpicture}
\end{minipage}
\caption{$\mathcal{F} = \{G_1\} \cup \{G_2\} \cup\{G_3,G_4\} \cup\{G_5,G_6,G_7\} \cup\{G_8\}$.} \label{F}
\end{center}
\end{figure}

The following lemma will help us to find a rainbow $C_{4}^{\text{multi}}$ in the edge-colored complete $3$-partite graphs and the idea comes from \cite{Alon}.

\begin{lemma}\label{lemma2}
Let $n_{1}\ge n_{2}\ge n_{3}\ge 1$. For an edge-colored $K_{n_{1},n_{2},n_{3}}$, if there is a rainbow copy of some graph in $\mathcal{F}$, then there is a rainbow copy of $C_{4}^{\text{multi}}$.
\end{lemma}

\begin{proof}

We separate the  proof into three cases.

{\bf Case 1.} An edge-colored $K_{n_1,n_2,n_3}$ contains a rainbow copy of $G_2$, $G_3$ or $G_4$.

\begin{figure}[h]
\centering
\begin{minipage}{4cm}
\begin{tikzpicture}[scale=0.8]
\draw[rotate around={40:(-12,5)}] (-12,5) arc(0:360: .5cm and 1.3cm);
\draw[rotate around={140:(-15,5)}] (-15,5) arc(0:360: .5cm and 1.3cm);
\draw (-12.2,3) arc(0:360: 1.3cm and 0.5cm);

\draw [fill](-12,4.5)circle [radius = 0.1];
\draw [fill](-12.5,5)circle [radius = 0.1];
\draw [fill](-15,4.5)circle [radius = 0.1];
\draw [fill](-14.5,5)circle [radius = 0.1];
\draw [fill](-13.5,3)circle [radius = 0.1];
\node[below] at(-11.8,4.5){$w_1$};
\node[below] at(-12.3,5){$w_2$};
\node[below] at(-15.1,4.5){$v_1$};
\node[below] at(-14.7,5){$v_2$};
\node[below] at(-13.5,3){$u$};

\draw  (-12,4.5) -- (-13.5,3);
\draw  (-12.5,5) -- (-13.5,3);
\draw  (-15,4.5) -- (-13.5,3);
\draw  (-14.5,5) -- (-13.5,3);
\draw  (-12,4.5) -- (-15,4.5);
\draw  (-12.5,5) -- (-14.5,5);
\draw[red] (-15,4.5) -- (-12.5,5); 
\node[right] at(-13.8,2){$G_2$};
\end{tikzpicture}
\end{minipage}
\begin{minipage}{8cm}
\begin{tikzpicture}[scale=0.8]
\draw[rotate around={40:(-17,0)}] (-17,0) arc(0:360: .5cm and 1.3cm);
\draw[rotate around={140:(-20,0)}] (-20,0) arc(0:360: .5cm and 1.3cm);
\draw (-17.2,-2) arc(0:360: 1.3cm and 0.5cm);

\draw [fill](-17,-0.5)circle [radius = 0.1];
\draw [fill](-20,-0.5)circle [radius = 0.1];
\draw [fill](-19.5,0)circle [radius = 0.1];
\draw [fill](-18,-2)circle [radius = 0.1];
\draw [fill](-18.5,-2)circle [radius = 0.1];
\draw [fill](-19,-2)circle [radius = 0.1];

\draw  (-20,-0.5) -- (-19,-2);
\draw  (-19.5,0) -- (-19,-2);
\draw  (-20,-0.5) -- (-18.5,-2);
\draw  (-18,-2) -- (-17,-0.5);
\draw  (-18.5,-2) -- (-17,-0.5);
\draw  (-18,-2) -- (-19.5,0);

\draw[red] (-19,-2) -- (-17,-0.5); 
\node[right] at(-18.8,-3){$G_3$};
\draw[rotate around={40:(-12,0)}] (-12,0) arc(0:360: .5cm and 1.3cm);
\draw[rotate around={140:(-15,0)}] (-15,0) arc(0:360: .5cm and 1.3cm);
\draw (-12.2,-2) arc(0:360: 1.3cm and 0.5cm);

\draw [fill](-12,-0.5)circle [radius = 0.1];
\draw [fill](-12.5,0)circle [radius = 0.1];
\draw [fill](-15,-0.5)circle [radius = 0.1];
\draw [fill](-14.5,0)circle [radius = 0.1];
\draw [fill](-13,-2)circle [radius = 0.1];
\draw [fill](-14,-2)circle [radius = 0.1];

\draw  (-12,-0.5) -- (-15,-0.5);
\draw  (-12.5,0) -- (-14.5,0);
\draw  (-12,-0.5) -- (-13,-2);
\draw  (-12.5,0) -- (-13,-2);
\draw  (-14.5,0) -- (-14,-2);
\draw  (-15,-0.5) -- (-14,-2);

\draw[red] (-15,-0.5) -- (-12.5,0); 
\node[right] at(-13.8,-3){$G_4$};
\end{tikzpicture}
\end{minipage}
\caption{}
\label{l5}
\end{figure}
Suppose there is a rainbow copy of $G_2$ in  $K_{n_1,n_2,n_3}$, See Figure \ref{l5}, then whatever the color of $v_{1}w_{2}$ is, at least one of  $v_1uv_2w_2v_1$ and $v_1w_2uw_1v_1$ is a rainbow  $C_{4}^{\text{multi}}$. Similarly, with the help of the red edge that showed in $G_3$ and $G_4$, see Figure \ref{l5}, one can easily find a rainbow copy of $C_{4}^{\text{multi}}$ if there is a rainbow copy of $G_3$ or $G_{4}$.


{\bf Case 2.} An edge-colored $K_{n_1,n_2,n_3}$ contains a rainbow copy of $G_5$.


Suppose there is a rainbow copy of $G_5$ in 
$K_{n_1,n_2,n_3}$, see Figure \ref{l6}. If $v_3w_3uw_2v_3$ is not rainbow, then $uw_3$ shares the same color with one of $v_3w_3$, $v_3w_2$ and $uw_2$. Hence, $uv_2w_3u\cup uv_1w_2u$ is a rainbow copy of $G_2$, by Case $1$, we can find a rainbow copy of $C_{4}^{\text{multi}}$.


\begin{figure}[h]
    \centering
\begin{tikzpicture}[scale=0.8]
\draw[rotate around={40:(-7,5)}] (-7,5) arc(0:360: .5cm and 1.3cm);
\draw[rotate around={140:(-10,5)}] (-10,5) arc(0:360: .5cm and 1.3cm);
\draw (-7.2,3) arc(0:360: 1.3cm and 0.5cm);

\draw [fill](-7,4)circle [radius = 0.1];
\draw [fill](-7.8,5)circle [radius = 0.1];
\draw [fill](-7.4,4.5)circle [radius = 0.1];
\draw [fill](-10,4)circle [radius = 0.1];
\draw [fill](-9.2,5)circle [radius = 0.1];
\draw [fill](-9.6,4.5)circle [radius = 0.1];
\draw [fill](-8.5,3)circle [radius = 0.1];

\node at(-10.3,4){$v_1$};
\node at(-10,  4.5){$v_2$};
\node at(-9.6, 5){$v_3$};
\node at(-7.3, 5){$w_3$};
\node at(-7, 4.5){$w_2$};
\node at(-6.6, 4){$w_1$};
\node at(-8.5, 2.7){$u$};

\draw  (-7.4,4.5) -- (-8.5,3);
\draw  (-9.6,4.5) -- (-8.5,3);
\draw  (-7,4) -- (-8.5,3);
\draw  (-10,4) -- (-8.5,3);
\draw  (-7.4,4.5) -- (-9.2,5);
\draw  (-9.6,4.5) -- (-7.8,5);
\draw  (-7,4) -- (-10,4);
\draw  (-7.8,5) -- (-9.2,5);

\draw[red]  (-8.5,3) -- (-7.8,5);
\node[right] at(-8.8,2){$G_5$};
\end{tikzpicture}
\caption{}
\label{l6}
\end{figure}

{\bf Case 3.} An edge-colored $K_{n_1,n_2,n_3}$ contains a rainbow copy of $G_6$, $G_7$ or $G_8$.

\begin{figure}[h]
    \centering
    \begin{minipage}{4cm}
\begin{tikzpicture}[scale=0.8]

\draw[rotate around={40:(-2,5)}] (-2,5) arc(0:360: .5cm and 1.3cm);
\draw[rotate around={140:(-5,5)}] (-5,5) arc(0:360: .5cm and 1.3cm);
\draw (-2.2,3) arc(0:360: 1.3cm and 0.5cm);

\draw [fill](-2,4.5)circle [radius = 0.1];
\draw [fill](-2.5,5)circle [radius = 0.1];
\draw [fill](-4,5.2)circle [radius = 0.1];
\draw [fill](-4.3,4.8)circle [radius = 0.1];
\draw [fill](-4.7,4.2)circle [radius = 0.1];
\draw [fill](-3.5,3)circle [radius = 0.1];
\draw [fill](-4,3)circle [radius = 0.1];

\node at(-5.1, 4.2){$v_1$};
\node at(-4.4, 5.2){$v_3$};
\node at(-4.7, 4.8){$v_2$};
\node at(-1.9, 4.1){$w_1$};
\node at(-2.2, 4.8){$w_2$};
\node at(-3.5, 2.7){$u_2$};
\node at(-4, 2.7){$u_1$};

\draw  (-4,3) -- (-4,5.2);
\draw  (-2.5,5) -- (-3.5,3);
\draw (-2.5,5) -- (-4,5.2);
\draw  (-4,3) -- (-4.3,4.8);
\draw  (-4.7,4.2) -- (-2,4.5);
\draw  [red](-3.5,3) -- (-2,4.5);
\draw  (-4,3) -- (-2,4.5);
\draw  (-4.7,4.2) -- (-4,3);
\draw  (-4.3,4.8) -- (-3.5,3);
\node[right] at(-3.8,2){$G_6$};
\end{tikzpicture}
\end{minipage}
\begin{minipage}{4cm}
\begin{tikzpicture}[scale=0.8]

\draw[rotate around={40:(-2,5)}] (-2,5) arc(0:360: .5cm and 1.3cm);
\draw[rotate around={140:(-5,5)}] (-5,5) arc(0:360: .5cm and 1.3cm);
\draw (-2.2,3) arc(0:360: 1.3cm and 0.5cm);

\draw [fill](-2,4.5)circle [radius = 0.1];
\draw [fill](-2.5,5)circle [radius = 0.1];
\draw [fill](-4,5.2)circle [radius = 0.1];
\draw [fill](-4.3,4.8)circle [radius = 0.1];
\draw [fill](-4.7,4.2)circle [radius = 0.1];
\draw [fill](-3.5,3)circle [radius = 0.1];
\draw [fill](-4,3)circle [radius = 0.1];

\draw  (-3.5,3) -- (-4,5.2);
\draw  (-2.5,5) -- (-4,3);
\draw (-2.5,5) -- (-4,5.2);

\draw  (-4,3) -- (-4.3,4.8);
\draw  (-4.7,4.2) -- (-2,4.5);
\draw  [red](-4.3,4.8) -- (-2,4.5);
\draw  (-4,3) -- (-2,4.5);
\draw  (-4.7,4.2) -- (-4,3);
\draw  (-4.3,4.8) -- (-3.5,3);
\node[right] at(-3.8,2){$G_7$};
\end{tikzpicture}
\end{minipage}
\begin{minipage}{4cm}
\begin{tikzpicture}[scale=0.8]
\draw[rotate around={40:(-7,-5)}] (-7,-5) arc(0:360: .5cm and 1.3cm);
\draw[rotate around={140:(-10,-5)}] (-10,-5) arc(0:360: .5cm and 1.3cm);
\draw (-7.2,-7) arc(0:360: 1.3cm and 0.5cm);

\draw [fill](-7,-6)circle [radius = 0.1];
\draw [fill](-7.8,-5)circle [radius = 0.1];
\draw [fill](-7.4,-5.5)circle [radius = 0.1];
\draw [fill](-10,-6)circle [radius = 0.1];
\draw [fill](-9.2,-5)circle [radius = 0.1];
\draw [fill](-9.6,-5.5)circle [radius = 0.1];
\draw [fill](-8,-7)circle [radius = 0.1];
\draw [fill](-9,-7)circle [radius = 0.1];

\draw  (-7.8,-5) -- (-9.6,-5.5);
\draw  (-7.4,-5.5) -- (-9.2,-5);
\draw  (-7,-6) -- (-9.2,-5);
\draw  (-7.8,-5) -- (-10,-6);
\draw  (-8,-7) -- (-9.6,-5.5);
\draw  [red](-7.4,-5.5) -- (-8,-7);
\draw  (-7.4,-5.5) -- (-9,-7);
\draw  (-7,-6) -- (-8,-7);
\draw  (-9,-7) -- (-10,-6);
\node[right] at(-8.8,-8){$G_8$};
\end{tikzpicture}
\end{minipage}
\caption{}
\label{l7}
\end{figure}

Suppose there is a rainbow copy of $G_6$ in 
$K_{n_1,n_2,n_3}$, see Figure \ref{l7}. If $v_2u_1w_1u_2v_2$ is not rainbow, then $u_2w_1$ shares the same color with one of $v_2u_1$, $u_1w_1$ and $u_2v_2$. Hence, $v_1u_1v_3w_2u_2w_1v_1$ is a rainbow copy of $G_4$, by Case $1$, we can find a rainbow copy of $C_{4}^{\text{multi}}$. Similarly, with the help of the red edge showed in $G_7$ and $G_8$, see Figure \ref{l7}, one can 
always find a rainbow copy of $C_{4}^{\text{multi}}$ if there is a rainbow copy of $G_7$ or $G_8$.
\end{proof}



Now we are able to prove Theorem \ref{arC4}.

\begin{proof} [Proof of Theorem \ref{arC4}]

{\bf Lower bound:} We color the edges of $K_{n_{1},n_{2},n_{3}}$ as follows.  First, color all edges between $V_{1}$ and $V_{2}$ rainbow. Second, for each vertex $v\in V_{3}$, color all the edges between $v$ and $V_{1}$ with one new distinct color. Finally, we assign a new color to all edges between $V_{2}$ and $V_{3}$. In such way, we use exactly $n_{1}n_{2}+n_{3}+1$ colors, and there is no rainbow $C_{4}^{\text{multi}}$.

{\bf Upper bound:} We prove the upper bound by induction on $n_{1}+n_{2}+n_{3}$. By Theorem \ref{exC4}, we have $\ar(K_{n_{1},n_{2},1}, C_{4}^{\text{multi}})\le \ex(K_{n_{1},n_{2},1}, C_{4}^{\text{multi}})=n_{1}n_{2}+2$, the conclusion holds for $n_{3}=1$. Let $n_{3}\ge 2$, suppose the conclusion holds for all integers less than $n_{1}+n_{2}+n_{3}$. We suppose there exists an $(n_{1}n_{2}+n_{3}+2)$-edge-coloring $c$ of $K_{n_{1},n_{2},n_{3}}$ such that there is no rainbow $C_{4}^{\text{multi}}$ in it. We take a representing subgraph $G$.

{\bf Claim 1.} $G$ contains two vertex-disjoint triangles.

{\bf Proof of Claim 1.} By Theorem \ref{exC3C4}, $\ex(K_{n_{1},n_{2},n_{3}}, \{C_{3}, C_{4}^{\text{multi}}\})=n_{1}n_{2}+n_{3}$. Since $e(G)=n_{1}n_{2}+n_{3}+2$ and $G$ contains no $C_{4}^{\text{multi}}$, $G$ contains at least two triangles $T_{1}$ and $T_{2}$. If $|V(T_{1})\cap V(T_{2})|=2$, then $T_{1}\cup T_{2}$ contains a $C_{4}^{\text{multi}}$, a contradiction. If $|V(T_{1})\cap V(T_{2})|=1$, then $T_{1}\cup T_{2}$ is a copy of $C_{3}\wedge C_{3}$. By Lemma \ref{lemma2}, we can find a rainbow $C_{4}^{\text{multi}}$, a contradiction. Thus, $T_{1}$ and $T_{2}$ are vertex-disjoint. \q

Let the two vertex-disjoint triangles be  $T_{1}=x_{1}y_{1}z_{1}x_{1}$ and $T_{2}=x_{2}y_{2}z_{2}x_{2}$, where $\{x_{1}, x_{2}\}\subseteq V_{1}$, $\{y_{1}, y_{2}\}\subseteq V_{2}$ and $\{z_{1}, z_{2}\}\subseteq V_{3}$. Denote $V_{0}=\{x_{1}, x_{2}, y_{1}, y_{2}, z_{1}, z_{2}\}$ and $U=(V_{1}\cup V_{2}\cup V_{3})\setminus V_{0}$.

{\bf Claim 2.} $e(G[V_{0}])\le 7$.

{\bf Proof of Claim 2.}  If $e(G[V_{0}])\ge 8$, then $e(V(T_{1}), V(T_{2}))\ge2$. Without loss of generality, assume that $x_{1}y_{2}\in E(G)$, we claim that $x_{1}z_{2}, x_{2}z_{1}, y_{1}z_{2}, y_{2}z_{1}\notin E(G)$, otherwise $x_{1}y_{2}x_{2}z_{2}x_{1}$, $x_{1}y_{2}x_{2}z_{1}x_{1}$, $x_{1}y_{2}z_{2}y_{1}x_{1}$ or $x_{1}y_{2}z_{1}y_{1}x_{1}$ would be a rainbow $C_{4}^{\text{multi}}$.
Thus, we have $x_{2}y_{1}\in E(G)$. We claim that $c(y_{1}z_{2})=c(y_{2}z_{2})$, otherwise at least one of $\{x_{1}y_{1}z_{2}y_{2}x_{1}, x_{2}y_{1}z_{2}y_{2}x_{2}\}$ is a rainbow $C_{4}^{\text{multi}}$. Thus, $G[V_{0}]-y_{2}z_{2}+y_{1}z_{2}$ is rainbow and contains a $C_{3}\wedge C_{3}$. By Lemma \ref{lemma2}, we find a rainbow $C_{4}^{\text{multi}}$, a contradiction. \q

If $U=\0$, that is $n_{1}=n_{2}=n_{3}=2$, then $8=e(G)=e(G[V_{0}])\le 7$, by Claim 2, a contradiction. Thus we may assume that $U\ne \0$.

{\bf Claim 3.} For all $v\in U$, $e(v,V_{0})\le 2.$

{\bf Proof of Claim 3.}  If there is a vertex $v\in U$, such that $e_{G}(v,V_{0})\ge 3$, then $G[V_{0}\cup \{v\}]$ contains a $C_{4}^{\text{multi}}$, a contradiction. \q

{\bf Claim 4.} $n_{3}\ge 3.$

{\bf Proof of Claim 4.} Suppose $n_{3}=2.$ Since $U\ne \0$, we have $n_{1}\ge 3=n_{3}+1$. If there is a vertex $v\in V_{1}$ such that $d(v)\le n_{2}$, then $e(G-v)= n_{1}n_{2}+n_{3}+2-d(v)\ge (n_{1}-1)n_{2}+n_{3}+2.$ By the induction hypothesis, we have $$|C(K_{n_{1},n_{2},n_{3}}-v)|\ge e(G-v)\ge (n_{1}-1)n_{2}+n_{3}+2= \ar(K_{n_{1}-1,n_{2},n_{3}}, C_{4}^{\text{multi}})+1,$$
thus $K_{n_{1},n_{2},n_{3}}-v$ contains a rainbow $C_{4}^{\text{multi}}$, a contradiction.
Thus we assume that $d(v)\ge n_{2}+1$ for all $v\in V_{1}$. By Claim 1, we have $e(V_{2}, V_{3})\ge 2$. Hence, we have $$e(G)=e(V_{1}, V_{2}\cup V_{3})+e(V_{2},V_{3})=\sum_{v\in V_{1}}d(v)+e(V_{2},V_{3})\ge n_{1}(n_{2}+1)+2=n_{1}n_{2}+n_{1}+2,$$
and this contradicts to the fact that $e(G)=n_{1}n_{2}+n_{3}+2$.\q

{\bf Claim 5.} $e(G[V_{0}])+e(V_{0}, U)\ge 2n_{1}+2n_{2}-1$.

{\bf Proof of Claim 5.} If $e(G[V_{0}])+e(V_{0}, U)\le 2n_{1}+2n_{2}-2$, then
$$e(G[U])=e(G)-(e(G[V_{0}])+e(V_{0}, U))\ge n_{1}n_{2}+n_{3}+2-(2n_{1}+2n_{2}-2)= (n_{1}-2)(n_{2}-2)+(n_{3}-2)+2.$$
By the induction hypothesis, we have
$$|C(K_{n_{1},n_{2},n_{3}}-V_{0})|\ge e(G[U])\ge  (n_{1}-2)(n_{2}-2)+(n_{3}-2)+2=\ar(K_{n_{1}-2,n_{2}-2,n_{3}-2},C_{4}^{\text{multi}})+1,$$
thus $K_{n_{1},n_{2},n_{3}}-V_{0}$ contains a rainbow $C_{4}^{\text{multi}}$, a contradiction. \q

Denote $U_{0}=\{v\in U: e(v,V_{0})=2\}$. By Claim 3, we have $e(U,V_{0})\le |U_{0}|+|U|$. By Claim 2, we just need to consider the following two cases.

{\bf Case 1.} $e(G[V_{0}])=7$.

By Claim 5, we have $e(U,V_{0})\ge 2n_{1}+2n_{2}-1-e(G[V_{0}])=2n_{1}+2n_{2}-8$.
Since $|U|=n_{1}+n_{2}+n_{3}-6$ and $e(U,V_{0})\le |U_{0}|+|U|$, we have $|U_{0}|\ge n_{1}+n_{2}-n_{3}-2\ge 1$. Let $v\in U_{0}$, then the orange edges in  $G[V_{0}\cup \{v\}]$ (see Figure \ref{l8}) forms one subgraph in $\mathcal{F}$ (see Figure \ref{F}). By Lemma \ref{lemma2}, there is a rainbow  $C_{4}^{\text{multi}}$, a contradiction.

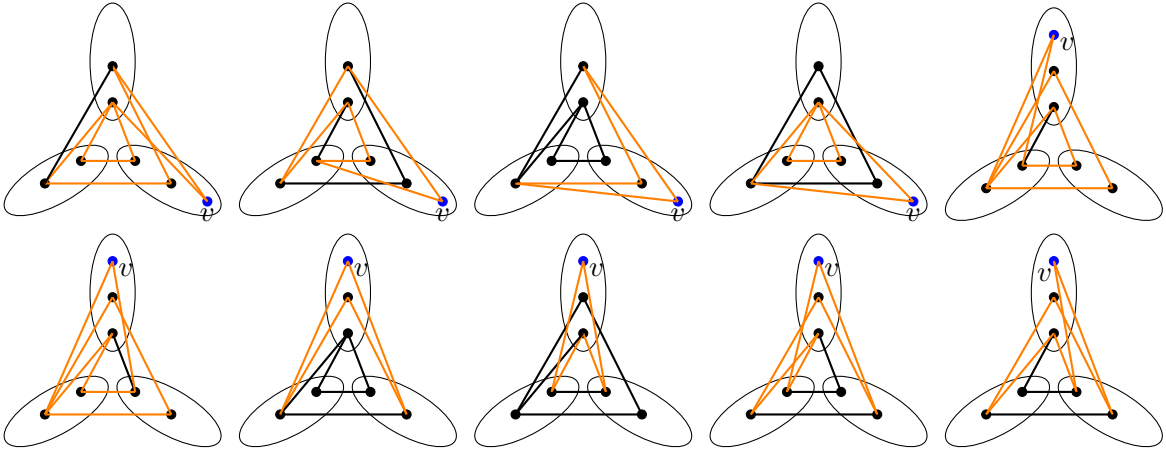
\begin{figure}[h] 
\begin{center}
\begin{minipage}{3cm}
\begin{tikzpicture}[scale=0.6]

\draw[rotate around={60:(-17,5)}] (-17,5) arc(0:360: .5cm and 1.3cm);
\draw[rotate around={120:(-20,5)}] (-20,5) arc(0:360: .5cm and 1.3cm);
\draw[rotate around={90:(-18.5,8.5)}](-18.5,8.5) arc(0:360: 1.3cm and 0.5cm);

\draw [fill](-18.5,7.1)circle [radius = 0.1];
\draw [fill](-18.5,6.3)circle [radius = 0.1];

\draw [fill](-17.2,4.5)circle [radius = 0.1];
\draw [fill](-18,5)circle [radius = 0.1];

\draw [fill](-20,4.5)circle [radius = 0.1];
\draw [fill](-19.2,5)circle [radius = 0.1];

\draw [blue,fill](-16.4,4.1)circle [radius = 0.1];

\node at(-16.4, 3.8) {$v$};

\draw [orange,thick] (-18.5,7.1) -- (-17.2,4.5);
\draw [orange,thick](-18.5,7.1) -- (-16.4,4.1);
\draw [orange,thick] (-16.4,4.1) -- (-18.5,6.3);
\draw [orange,thick] (-17.2,4.5) -- (-20,4.5);
\draw [orange,thick] (-18.5,6.3) -- (-20,4.5);
\draw [orange,thick] (-18.5,6.3) -- (-18,5);
\draw [orange,thick] (-18.5,6.3) -- (-19.2,5);
\draw [orange,thick] (-18,5) -- (-19.2,5);
\draw [thick] (-18.5,7.1) -- (-20,4.5);
\end{tikzpicture}
\end{minipage}
\begin{minipage}{3cm}
\begin{tikzpicture}[scale=0.6]

\draw[rotate around={60:(-17,5)}] (-17,5) arc(0:360: .5cm and 1.3cm);
\draw[rotate around={120:(-20,5)}] (-20,5) arc(0:360: .5cm and 1.3cm);
\draw[rotate around={90:(-18.5,8.5)}](-18.5,8.5) arc(0:360: 1.3cm and 0.5cm);

\draw [fill](-18.5,7.1)circle [radius = 0.1];
\draw [fill](-18.5,6.3)circle [radius = 0.1];

\draw [fill](-17.2,4.5)circle [radius = 0.1];
\draw [fill](-18,5)circle [radius = 0.1];

\draw [fill](-20,4.5)circle [radius = 0.1];
\draw [fill](-19.2,5)circle [radius = 0.1];

\draw [blue,fill](-16.4,4.1)circle [radius = 0.1];

\node at(-16.4, 3.8) {$v$};

\draw [thick] (-18.5,7.1) -- (-17.2,4.5);
\draw [orange,thick] (-18.5,7.1) -- (-20,4.5);
\draw  [thick](-17.2,4.5) -- (-20,4.5);

\draw [orange,thick](-18.5,6.3) -- (-18,5);
\draw [thick] (-18.5,6.3) -- (-19.2,5);
\draw  [orange,thick](-18,5) -- (-19.2,5);

\draw [orange,thick](-18.5,7.1) -- (-16.4,4.1);
\draw [orange,thick] (-18.5,6.3) -- (-20,4.5);
\draw [orange,thick] (-16.4,4.1) -- (-19.2,5);
\end{tikzpicture}
\end{minipage}
\begin{minipage}{3cm}
\begin{tikzpicture}[scale=0.6]

\draw[rotate around={60:(-17,5)}] (-17,5) arc(0:360: .5cm and 1.3cm);
\draw[rotate around={120:(-20,5)}] (-20,5) arc(0:360: .5cm and 1.3cm);
\draw[rotate around={90:(-18.5,8.5)}](-18.5,8.5) arc(0:360: 1.3cm and 0.5cm);

\draw [fill](-18.5,7.1)circle [radius = 0.1];
\draw [fill](-18.5,6.3)circle [radius = 0.1];

\draw [fill](-17.2,4.5)circle [radius = 0.1];
\draw [fill](-18,5)circle [radius = 0.1];

\draw [fill](-20,4.5)circle [radius = 0.1];
\draw [fill](-19.2,5)circle [radius = 0.1];

\draw [blue,fill](-16.4,4.1)circle [radius = 0.1];

\node at(-16.4, 3.8) {$v$};

\draw [orange, thick] (-18.5,7.1) -- (-17.2,4.5);
\draw [thick](-18.5,7.1) -- (-20,4.5);
\draw  [orange, thick](-17.2,4.5) -- (-20,4.5);

\draw [thick] (-18.5,6.3) -- (-18,5);
\draw  [thick] (-18.5,6.3) -- (-19.2,5);
\draw [thick] (-18,5) -- (-19.2,5);

\draw [orange, thick](-18.5,7.1) -- (-16.4,4.1);
\draw [black, thick] (-18.5,6.3) -- (-20,4.5);
\draw [orange, thick] (-16.4,4.1) -- (-20,4.5);
\end{tikzpicture}
\end{minipage}
\begin{minipage}{3cm}
\begin{tikzpicture}[scale=0.6]

\draw[rotate around={60:(-17,5)}] (-17,5) arc(0:360: .5cm and 1.3cm);
\draw[rotate around={120:(-20,5)}] (-20,5) arc(0:360: .5cm and 1.3cm);
\draw[rotate around={90:(-18.5,8.5)}](-18.5,8.5) arc(0:360: 1.3cm and 0.5cm);

\draw [fill](-18.5,7.1)circle [radius = 0.1];
\draw [fill](-18.5,6.3)circle [radius = 0.1];

\draw [fill](-17.2,4.5)circle [radius = 0.1];
\draw [fill](-18,5)circle [radius = 0.1];

\draw [fill](-20,4.5)circle [radius = 0.1];
\draw [fill](-19.2,5)circle [radius = 0.1];

\draw [blue,fill](-16.4,4.1)circle [radius = 0.1];

\node at(-16.4, 3.8) {$v$};

\draw [ thick] (-18.5,7.1) -- (-17.2,4.5);
\draw [thick] (-18.5,7.1) -- (-20,4.5);
\draw [thick] (-17.2,4.5) -- (-20,4.5);

\draw [orange,thick](-18.5,6.3) -- (-18,5);
\draw  [orange,thick](-18.5,6.3) -- (-19.2,5);
\draw  [orange,thick](-18,5) -- (-19.2,5);

\draw [orange,thick](-18.5,6.3) -- (-16.4,4.1);
\draw  [orange,thick](-18.5,6.3) -- (-20,4.5);
\draw [orange,thick] (-16.4,4.1) -- (-20,4.5);
\end{tikzpicture}
\end{minipage}
\begin{minipage}{3cm}
\begin{tikzpicture}[scale=0.6]

\draw[rotate around={60:(-17,5)}] (-17,5) arc(0:360: .5cm and 1.3cm);
\draw[rotate around={120:(-20,5)}] (-20,5) arc(0:360: .5cm and 1.3cm);
\draw[rotate around={90:(-18.5,8.5)}](-18.5,8.5) arc(0:360: 1.3cm and 0.5cm);

\draw [fill](-18.5,7.1)circle [radius = 0.1];
\draw [fill](-18.5,6.3)circle [radius = 0.1];

\draw [fill](-17.2,4.5)circle [radius = 0.1];
\draw [fill](-18,5)circle [radius = 0.1];

\draw [fill](-20,4.5)circle [radius = 0.1];
\draw [fill](-19.2,5)circle [radius = 0.1];

\draw [blue,fill](-18.5,7.9)circle [radius = 0.1];

\node at(-18.2, 7.7) {$v$};

\draw  [orange,thick](-18.5,7.1) -- (-17.2,4.5);
\draw  [orange,thick](-18.5,7.1) -- (-20,4.5);
\draw  [orange,thick](-17.2,4.5) -- (-20,4.5);

\draw [orange,thick](-18.5,6.3) -- (-18,5);
\draw [ thick] (-18.5,6.3) -- (-19.2,5);
\draw  [orange,thick](-18,5) -- (-19.2,5);

\draw [orange,thick](-19.2,5) -- (-18.5,7.9);
\draw  [orange,thick](-18.5,6.3) -- (-20,4.5);
\draw [orange,thick] (-18.5,7.9) -- (-20,4.5);
\end{tikzpicture}
\end{minipage}\\
\begin{minipage}{3cm}
\begin{tikzpicture}[scale=0.6]

\draw[rotate around={60:(-17,5)}] (-17,5) arc(0:360: .5cm and 1.3cm);
\draw[rotate around={120:(-20,5)}] (-20,5) arc(0:360: .5cm and 1.3cm);
\draw[rotate around={90:(-18.5,8.5)}](-18.5,8.5) arc(0:360: 1.3cm and 0.5cm);

\draw [fill](-18.5,7.1)circle [radius = 0.1];
\draw [fill](-18.5,6.3)circle [radius = 0.1];

\draw [fill](-17.2,4.5)circle [radius = 0.1];
\draw [fill](-18,5)circle [radius = 0.1];

\draw [fill](-20,4.5)circle [radius = 0.1];
\draw [fill](-19.2,5)circle [radius = 0.1];

\draw [blue,fill](-18.5,7.9)circle [radius = 0.1];

\node at(-18.2, 7.7) {$v$};

\draw  [orange,thick](-18.5,7.1) -- (-17.2,4.5);
\draw  [orange,thick](-18.5,7.1) -- (-20,4.5);
\draw  [orange,thick](-17.2,4.5) -- (-20,4.5);

\draw [ thick] (-18.5,6.3) -- (-18,5);
\draw  [orange,thick](-18.5,6.3) -- (-19.2,5);
\draw  [orange,thick](-18,5) -- (-19.2,5);

\draw [orange,thick](-18,5) -- (-18.5,7.9);
\draw  [orange,thick](-18.5,6.3) -- (-20,4.5);
\draw [orange,thick] (-18.5,7.9) -- (-20,4.5);
\end{tikzpicture}
\end{minipage}
\begin{minipage}{3cm}
\begin{tikzpicture}[scale=0.6]

\draw[rotate around={60:(-17,5)}] (-17,5) arc(0:360: .5cm and 1.3cm);
\draw[rotate around={120:(-20,5)}] (-20,5) arc(0:360: .5cm and 1.3cm);
\draw[rotate around={90:(-18.5,8.5)}](-18.5,8.5) arc(0:360: 1.3cm and 0.5cm);

\draw [fill](-18.5,7.1)circle [radius = 0.1];
\draw [fill](-18.5,6.3)circle [radius = 0.1];

\draw [fill](-17.2,4.5)circle [radius = 0.1];
\draw [fill](-18,5)circle [radius = 0.1];

\draw [fill](-20,4.5)circle [radius = 0.1];
\draw [fill](-19.2,5)circle [radius = 0.1];

\draw [blue,fill](-18.5,7.9)circle [radius = 0.1];
\node at(-18.2, 7.7) {$v$};

\draw [orange,thick](-18.5,7.1) -- (-17.2,4.5);
\draw  [orange,thick](-18.5,7.1) -- (-20,4.5);
\draw [thick] (-17.2,4.5) -- (-20,4.5);

\draw [thick](-18.5,6.3) -- (-18,5);
\draw  [thick](-18.5,6.3) -- (-19.2,5);
\draw [thick] (-18,5) -- (-19.2,5);

\draw [orange,thick](-17.2,4.5) -- (-18.5,7.9);
\draw [thick] (-18.5,6.3) -- (-20,4.5);
\draw [orange,thick] (-18.5,7.9) -- (-20,4.5);
\end{tikzpicture}
\end{minipage}
\begin{minipage}{3cm}
\begin{tikzpicture}[scale=0.6]

\draw[rotate around={60:(-17,5)}] (-17,5) arc(0:360: .5cm and 1.3cm);
\draw[rotate around={120:(-20,5)}] (-20,5) arc(0:360: .5cm and 1.3cm);
\draw[rotate around={90:(-18.5,8.5)}](-18.5,8.5) arc(0:360: 1.3cm and 0.5cm);

\draw [fill](-18.5,7.1)circle [radius = 0.1];
\draw [fill](-18.5,6.3)circle [radius = 0.1];

\draw [fill](-17.2,4.5)circle [radius = 0.1];
\draw [fill](-18,5)circle [radius = 0.1];

\draw [fill](-20,4.5)circle [radius = 0.1];
\draw [fill](-19.2,5)circle [radius = 0.1];

\draw [blue,fill](-18.5,7.9)circle [radius = 0.1];

\node at(-18.2, 7.7) {$v$};

\draw [thick](-18.5,7.1) -- (-17.2,4.5);
\draw  [thick](-18.5,7.1) -- (-20,4.5);
\draw  [thick](-17.2,4.5) -- (-20,4.5);

\draw [orange,thick](-18.5,6.3) -- (-18,5);
\draw   [orange,thick](-18.5,6.3) -- (-19.2,5);
\draw [thick] (-18,5) -- (-19.2,5);

\draw [orange,thick](-18,5) -- (-18.5,7.9);
\draw  [thick](-18.5,6.3) -- (-20,4.5);
\draw [orange,thick] (-18.5,7.9) -- (-19.2,5);
\end{tikzpicture}
\end{minipage}
\begin{minipage}{3cm}
\begin{tikzpicture}[scale=0.6]

\draw[rotate around={60:(-17,5)}] (-17,5) arc(0:360: .5cm and 1.3cm);
\draw[rotate around={120:(-20,5)}] (-20,5) arc(0:360: .5cm and 1.3cm);
\draw[rotate around={90:(-18.5,8.5)}](-18.5,8.5) arc(0:360: 1.3cm and 0.5cm);

\draw [fill](-18.5,7.1)circle [radius = 0.1];
\draw [fill](-18.5,6.3)circle [radius = 0.1];

\draw [fill](-17.2,4.5)circle [radius = 0.1];
\draw [fill](-18,5)circle [radius = 0.1];

\draw [fill](-20,4.5)circle [radius = 0.1];
\draw [fill](-19.2,5)circle [radius = 0.1];

\draw [blue,fill](-18.5,7.9)circle [radius = 0.1];

\node at(-18.2, 7.7) {$v$};

\draw    [orange,thick](-18.5,7.1) -- (-17.2,4.5);
\draw    [orange,thick](-18.5,7.1) -- (-20,4.5);
\draw [thick] (-17.2,4.5) -- (-20,4.5);

\draw [thick] (-18.5,6.3) -- (-18,5);
\draw [orange,thick] (-18.5,6.3) -- (-19.2,5);
\draw [thick] (-18,5) -- (-19.2,5);

\draw [orange,thick](-19.2,5) -- (-18.5,7.9);
\draw [orange,thick] (-18.5,6.3) -- (-20,4.5);
\draw [orange,thick] (-18.5,7.9) -- (-17.2,4.5);
\end{tikzpicture}
\end{minipage}
\begin{minipage}{3cm}
\begin{tikzpicture}[scale=0.6]

\draw[rotate around={60:(-17,5)}] (-17,5) arc(0:360: .5cm and 1.3cm);
\draw[rotate around={120:(-20,5)}] (-20,5) arc(0:360: .5cm and 1.3cm);
\draw[rotate around={90:(-18.5,8.5)}](-18.5,8.5) arc(0:360: 1.3cm and 0.5cm);

\draw [fill](-18.5,7.1)circle [radius = 0.1];
\draw [fill](-18.5,6.3)circle [radius = 0.1];

\draw [fill](-17.2,4.5)circle [radius = 0.1];
\draw [fill](-18,5)circle [radius = 0.1];

\draw [fill](-20,4.5)circle [radius = 0.1];
\draw [fill](-19.2,5)circle [radius = 0.1];

\draw [blue,fill](-18.5,7.9)circle [radius = 0.1];

\node at(-18.7, 7.6) {$v$};

\draw[orange,thick](-18.5,7.1) -- (-17.2,4.5);
\draw  [orange,thick](-18.5,7.1) -- (-20,4.5);
\draw [thick]  (-17.2,4.5) -- (-20,4.5);

\draw [orange,thick](-18.5,6.3) -- (-18,5);
\draw  [thick] (-18.5,6.3) -- (-19.2,5);
\draw  [thick](-18,5) -- (-19.2,5);

\draw [orange,thick](-17.2,4.5) -- (-18.5,7.9);
\draw  [orange,thick](-18.5,6.3) -- (-20,4.5);
\draw [orange,thick] (-18.5,7.9) -- (-18,5);
\end{tikzpicture}
\end{minipage}
\caption{ Illustration of
$G[V_{0}\cup \{v\}]$.} 
\label{l8}

\end{center}
\end{figure}

{\bf Case 2.} $e(G[V_{0}])=6$.

\begin{figure}[h] 
\begin{center}
\begin{minipage}{3cm}
\begin{tikzpicture}[scale=0.6]

\draw[rotate around={60:(-17,5)}] (-17,5) arc(0:360: .5cm and 1.3cm);
\draw[rotate around={120:(-20,5)}] (-20,5) arc(0:360: .5cm and 1.3cm);
\draw[rotate around={90:(-18.5,8.5)}](-18.5,8.5) arc(0:360: 1.3cm and 0.5cm);

\draw [fill](-18.5,7.1)circle [radius = 0.1];
\draw [fill](-18.5,6.3)circle [radius = 0.1];

\draw [fill](-17.2,4.5)circle [radius = 0.1];
\draw [fill](-18,5)circle [radius = 0.1];

\draw [fill](-20,4.5)circle [radius = 0.1];
\draw [fill](-19.2,5)circle [radius = 0.1];

\draw [blue,fill](-18.5,7.9)circle [radius = 0.1];
\draw [blue,fill](-16.4,4)circle [radius = 0.1];

\node at(-18,7.8) {$v_1$};

\node at(-16.4,3.6) {$v_2$};

\draw [thick] (-18.5,7.1) -- (-17.2,4.5);
\draw  [orange,thick](-18.5,7.1) -- (-20,4.5);
\draw [thick] (-17.2,4.5) -- (-20,4.5);

\draw [thick](-18.5,6.3) -- (-18,5);
\draw  [orange,thick](-18.5,6.3) -- (-19.2,5);
\draw [thick] (-18,5) -- (-19.2,5);

\draw  [orange,thick](-19.2,5) -- (-18.5,7.9);
\draw  [orange,thick] (-18.5,7.9) -- (-20,4.5);

\draw   [orange,thick](-18.5,6.3) -- (-16.4,4);
\draw   [orange,thick](-18.5,7.1) -- (-16.4,4);
\end{tikzpicture}
\end{minipage}
\begin{minipage}{3cm}
\begin{tikzpicture}[scale=0.6]

\draw[rotate around={60:(-17,5)}] (-17,5) arc(0:360: .5cm and 1.3cm);
\draw[rotate around={120:(-20,5)}] (-20,5) arc(0:360: .5cm and 1.3cm);
\draw[rotate around={90:(-18.5,8.5)}](-18.5,8.5) arc(0:360: 1.3cm and 0.5cm);

\draw [fill](-18.5,7.1)circle [radius = 0.1];
\draw [fill](-18.5,6.3)circle [radius = 0.1];

\draw [fill](-17.2,4.5)circle [radius = 0.1];
\draw [fill](-18,5)circle [radius = 0.1];

\draw [fill](-20,4.5)circle [radius = 0.1];
\draw [fill](-19.2,5)circle [radius = 0.1];

\draw [blue,fill](-18.5,7.9)circle [radius = 0.1];
\draw [blue,fill](-16.4,4)circle [radius = 0.1];

\node at(-18,7.8) {$v_1$};

\node at(-16.4,3.6) {$v_2$};
\draw [thick] (-18.5,7.1) -- (-17.2,4.5);
\draw [thick] (-18.5,7.1) -- (-20,4.5);
\draw [thick] (-17.2,4.5) -- (-20,4.5);

\draw [thick](-18.5,6.3) -- (-18,5);
\draw [thick] (-18.5,6.3) -- (-19.2,5);
\draw [thick] (-18,5) -- (-19.2,5);

\draw [orange,thick](-19.2,5) -- (-18.5,7.9);
\draw [orange,thick](-18.5,7.9) -- (-20,4.5);

\draw  [orange,thick](-20,4.5) -- (-16.4,4);
\draw [orange,thick](-19.2,5) -- (-16.4,4);
\end{tikzpicture}
\end{minipage}
\begin{minipage}{3cm}
\begin{tikzpicture}[scale=0.6]

\draw[rotate around={60:(-17,5)}] (-17,5) arc(0:360: .5cm and 1.3cm);
\draw[rotate around={120:(-20,5)}] (-20,5) arc(0:360: .5cm and 1.3cm);
\draw[rotate around={90:(-18.5,8.5)}](-18.5,8.5) arc(0:360: 1.3cm and 0.5cm);

\draw [fill](-18.5,7.1)circle [radius = 0.1];
\draw [fill](-18.5,6.3)circle [radius = 0.1];

\draw [fill](-17.2,4.5)circle [radius = 0.1];
\draw [fill](-18,5)circle [radius = 0.1];

\draw [fill](-20,4.5)circle [radius = 0.1];
\draw [fill](-19.2,5)circle [radius = 0.1];

\draw [blue,fill](-18.5,7.9)circle [radius = 0.1];
\draw [blue,fill](-16.4,4)circle [radius = 0.1];

\node at(-18,7.8) {$v_1$};

\node at(-16.4,3.6) {$v_2$};

\draw [thick] (-18.5,7.1) -- (-17.2,4.5);
\draw  [orange,thick](-18.5,7.1) -- (-20,4.5);
\draw [thick] (-17.2,4.5) -- (-20,4.5);

\draw  [orange,thick](-18.5,6.3) -- (-18,5);
\draw  [orange,thick](-18.5,6.3) -- (-19.2,5);
\draw  [orange,thick](-18,5) -- (-19.2,5);

\draw  [orange,thick](-19.2,5) -- (-18.5,7.9);
\draw   [orange,thick](-18.5,7.9) -- (-20,4.5);

\draw  [orange,thick](-18.5,7.1) -- (-16.4,4);
\draw  [orange,thick](-19.2,5) -- (-16.4,4);
\end{tikzpicture}
\end{minipage}
\begin{minipage}{3cm}
\begin{tikzpicture}[scale=0.6]

\draw[rotate around={60:(-17,5)}] (-17,5) arc(0:360: .5cm and 1.3cm);
\draw[rotate around={120:(-20,5)}] (-20,5) arc(0:360: .5cm and 1.3cm);
\draw[rotate around={90:(-18.5,8.5)}](-18.5,8.5) arc(0:360: 1.3cm and 0.5cm);

\draw [fill](-18.5,7.1)circle [radius = 0.1];
\draw [fill](-18.5,6.3)circle [radius = 0.1];

\draw [fill](-17.2,4.5)circle [radius = 0.1];
\draw [fill](-18,5)circle [radius = 0.1];

\draw [fill](-20,4.5)circle [radius = 0.1];
\draw [fill](-19.2,5)circle [radius = 0.1];

\draw [blue,fill](-18.5,7.9)circle [radius = 0.1];
\draw [blue,fill](-16.4,4)circle [radius = 0.1];

\node at(-18,7.8) {$v_1$};

\node at(-16.4,3.6) {$v_2$};

\draw [thick] (-18.5,7.1) -- (-17.2,4.5);
\draw [thick] (-18.5,7.1) -- (-20,4.5);
\draw [thick] (-17.2,4.5) -- (-20,4.5);

\draw [orange,thick](-18.5,6.3) -- (-18,5);
\draw [thick] (-18.5,6.3) -- (-19.2,5);
\draw  [orange,thick](-18,5) -- (-19.2,5);

\draw[thick] (-19.2,5) -- (-18.5,7.9);
\draw [thick] (-18.5,7.9) -- (-20,4.5);

\draw  [orange,thick](-18.5,6.3) -- (-16.4,4);
\draw  [orange,thick](-19.2,5) -- (-16.4,4);
\end{tikzpicture}
\end{minipage}
\begin{minipage}{3cm}
\begin{tikzpicture}[scale=0.6]

\draw[rotate around={60:(-17,5)}] (-17,5) arc(0:360: .5cm and 1.3cm);
\draw[rotate around={120:(-20,5)}] (-20,5) arc(0:360: .5cm and 1.3cm);
\draw[rotate around={90:(-18.5,8.5)}](-18.5,8.5) arc(0:360: 1.3cm and 0.5cm);

\draw [fill](-18.5,7.1)circle [radius = 0.1];
\draw [fill](-18.5,6.3)circle [radius = 0.1];

\draw [fill](-17.2,4.5)circle [radius = 0.1];
\draw [fill](-18,5)circle [radius = 0.1];

\draw [fill](-20,4.5)circle [radius = 0.1];
\draw [fill](-19.2,5)circle [radius = 0.1];

\draw [blue,fill](-18.5,7.9)circle [radius = 0.1];
\draw [blue,fill](-16.4,4)circle [radius = 0.1];
\node at(-18,7.8) {$v_1$};

\node at(-16.4,3.6) {$v_2$};

\draw [thick] (-18.5,7.1) -- (-17.2,4.5);
\draw  [orange,thick](-18.5,7.1) -- (-20,4.5);
\draw [thick] (-17.2,4.5) -- (-20,4.5);

\draw [orange,thick](-18.5,6.3) -- (-18,5);
\draw [thick] (-18.5,6.3) -- (-19.2,5);
\draw [thick] (-18,5) -- (-19.2,5);

\draw  [orange,thick](-18,5) -- (-18.5,7.9);
\draw  [orange,thick](-18.5,7.9) -- (-20,4.5);

\draw [orange,thick](-18.5,6.3) -- (-16.4,4);
\draw  [orange,thick](-18.5,7.1) -- (-16.4,4);
\end{tikzpicture}
\end{minipage}\\
\begin{minipage}{2.5cm}
\begin{tikzpicture}[scale=0.5]

\draw[rotate around={60:(-17,5)}] (-17,5) arc(0:360: .5cm and 1.3cm);
\draw[rotate around={120:(-20,5)}] (-20,5) arc(0:360: .5cm and 1.3cm);
\draw[rotate around={90:(-18.5,8.5)}](-18.5,8.5) arc(0:360: 1.3cm and 0.5cm);

\draw [fill](-18.5,7.1)circle [radius = 0.1];
\draw [fill](-18.5,6.3)circle [radius = 0.1];

\draw [fill](-17.2,4.5)circle [radius = 0.1];
\draw [fill](-18,5)circle [radius = 0.1];

\draw [fill](-20,4.5)circle [radius = 0.1];
\draw [fill](-19.2,5)circle [radius = 0.1];

\draw [blue,fill](-18.5,7.9)circle [radius = 0.1];
\draw [blue,fill](-16.4,4)circle [radius = 0.1];
\node at(-18,7.8) {$v_1$};

\node at(-16.4,3.6) {$v_2$};

\draw [thick] (-18.5,7.1) -- (-17.2,4.5);
\draw [orange,thick](-18.5,7.1) -- (-20,4.5);
\draw [thick] (-17.2,4.5) -- (-20,4.5);

\draw [thick](-18.5,6.3) -- (-18,5);
\draw  [thick](-18.5,6.3) -- (-19.2,5);
\draw  [orange,thick] (-18,5) -- (-19.2,5);

\draw  [orange,thick](-18,5) -- (-18.5,7.9);
\draw   [orange,thick](-18.5,7.9) -- (-20,4.5);

\draw   [orange,thick](-19.2,5) -- (-16.4,4);
\draw   [orange,thick](-18.5,7.1) -- (-16.4,4);
\end{tikzpicture}
\end{minipage}
\begin{minipage}{2.5cm}
\begin{tikzpicture}[scale=0.5]

\draw[rotate around={60:(-17,5)}] (-17,5) arc(0:360: .5cm and 1.3cm);
\draw[rotate around={120:(-20,5)}] (-20,5) arc(0:360: .5cm and 1.3cm);
\draw[rotate around={90:(-18.5,8.5)}](-18.5,8.5) arc(0:360: 1.3cm and 0.5cm);

\draw [fill](-18.5,7.1)circle [radius = 0.1];
\draw [fill](-18.5,6.3)circle [radius = 0.1];

\draw [fill](-17.2,4.5)circle [radius = 0.1];
\draw [fill](-18,5)circle [radius = 0.1];

\draw [fill](-20,4.5)circle [radius = 0.1];
\draw [fill](-19.2,5)circle [radius = 0.1];

\draw [blue,fill](-18.5,7.9)circle [radius = 0.1];
\draw [blue,fill](-16.4,4)circle [radius = 0.1];

\node at(-18,7.8) {$v_1$};

\node at(-16.4,3.6) {$v_2$};

\draw [orange,thick](-18.5,7.1) -- (-17.2,4.5);
\draw  [orange,thick](-18.5,7.1) -- (-20,4.5);
\draw  [orange,thick](-17.2,4.5) -- (-20,4.5);

\draw [orange,thick](-18.5,6.3) -- (-18,5);
\draw [thick] (-18.5,6.3) -- (-19.2,5);
\draw [thick] (-18,5) -- (-19.2,5);

\draw [orange,thick](-18,5) -- (-18.5,7.9);
\draw [orange,thick](-18.5,7.9) -- (-20,4.5);

\draw [orange,thick](-20,4.5) -- (-16.4,4);
\draw [orange,thick](-18.5,6.3) -- (-16.4,4);
\end{tikzpicture}
\end{minipage}
\begin{minipage}{2.5cm}
\begin{tikzpicture}[scale=0.5]

\draw[rotate around={60:(-17,5)}] (-17,5) arc(0:360: .5cm and 1.3cm);
\draw[rotate around={120:(-20,5)}] (-20,5) arc(0:360: .5cm and 1.3cm);
\draw[rotate around={90:(-18.5,8.5)}](-18.5,8.5) arc(0:360: 1.3cm and 0.5cm);

\draw [fill](-18.5,7.1)circle [radius = 0.1];
\draw [fill](-18.5,6.3)circle [radius = 0.1];

\draw [fill](-17.2,4.5)circle [radius = 0.1];
\draw [fill](-18,5)circle [radius = 0.1];

\draw [fill](-20,4.5)circle [radius = 0.1];
\draw [fill](-19.2,5)circle [radius = 0.1];

\draw [blue,fill](-18.5,7.9)circle [radius = 0.1];
\draw [blue,fill](-16.4,4)circle [radius = 0.1];

\node at(-18,7.8) {$v_1$};

\node at(-16.4,3.6) {$v_2$};

\draw [thick] (-18.5,7.1) -- (-17.2,4.5);
\draw  [thick](-18.5,7.1) -- (-20,4.5);
\draw [thick] (-17.2,4.5) -- (-20,4.5);

\draw [orange,thick](-18.5,6.3) -- (-18,5);
\draw [thick] (-18.5,6.3) -- (-19.2,5);
\draw [orange,thick](-18,5) -- (-19.2,5);

\draw [thick](-18,5) -- (-18.5,7.9);
\draw [thick] (-18.5,7.9) -- (-20,4.5);

\draw  [orange,thick](-19.2,5) -- (-16.4,4);
\draw  [orange,thick](-18.5,6.3) -- (-16.4,4);
\end{tikzpicture}
\end{minipage}
\begin{minipage}{2.5cm}
\begin{tikzpicture}[scale=0.5]

\draw[rotate around={60:(-17,5)}] (-17,5) arc(0:360: .5cm and 1.3cm);
\draw[rotate around={120:(-20,5)}] (-20,5) arc(0:360: .5cm and 1.3cm);
\draw[rotate around={90:(-18.5,8.5)}](-18.5,8.5) arc(0:360: 1.3cm and 0.5cm);

\draw [fill](-18.5,7.1)circle [radius = 0.1];
\draw [fill](-18.5,6.3)circle [radius = 0.1];

\draw [fill](-17.2,4.5)circle [radius = 0.1];
\draw [fill](-18,5)circle [radius = 0.1];

\draw [fill](-20,4.5)circle [radius = 0.1];
\draw [fill](-19.2,5)circle [radius = 0.1];

\draw [blue,fill](-18.5,7.9)circle [radius = 0.1];
\draw [blue,fill](-16.4,4)circle [radius = 0.1];

\node at(-18,7.8) {$v_1$};

\node at(-16.4,3.6) {$v_2$};

\draw   [orange,thick](-18.5,7.1) -- (-17.2,4.5);
\draw [thick] (-18.5,7.1) -- (-20,4.5);
\draw  [orange,thick](-17.2,4.5) -- (-20,4.5);

\draw [thick](-18.5,6.3) -- (-18,5);
\draw [thick] (-18.5,6.3) -- (-19.2,5);
\draw [thick] (-18,5) -- (-19.2,5);

\draw [thick](-18,5) -- (-18.5,7.9);
\draw  [thick](-18.5,7.9) -- (-20,4.5);

\draw [orange,thick](-20,4.5) -- (-16.4,4);
\draw  [orange,thick](-18.5,7.1) -- (-16.4,4);
\end{tikzpicture}
\end{minipage}
\begin{minipage}{2.5cm}
\begin{tikzpicture}[scale=0.5]

\draw[rotate around={60:(-17,5)}] (-17,5) arc(0:360: .5cm and 1.3cm);
\draw[rotate around={120:(-20,5)}] (-20,5) arc(0:360: .5cm and 1.3cm);
\draw[rotate around={90:(-18.5,8.5)}](-18.5,8.5) arc(0:360: 1.3cm and 0.5cm);

\draw [fill](-18.5,7.1)circle [radius = 0.1];
\draw [fill](-18.5,6.3)circle [radius = 0.1];

\draw [fill](-17.2,4.5)circle [radius = 0.1];
\draw [fill](-18,5)circle [radius = 0.1];

\draw [fill](-20,4.5)circle [radius = 0.1];
\draw [fill](-19.2,5)circle [radius = 0.1];

\draw [blue,fill](-18.5,7.9)circle [radius = 0.1];
\draw [blue,fill](-16.4,4)circle [radius = 0.1];

\node at(-18,7.8) {$v_1$};

\node at(-16.4,3.6) {$v_2$};

\draw  [thick](-18.5,7.1) -- (-17.2,4.5);
\draw  [orange,thick](-18.5,7.1) -- (-20,4.5);
\draw  [orange,thick](-17.2,4.5) -- (-20,4.5);

\draw [thick](-18.5,6.3) -- (-18,5);
\draw[orange,thick](-18.5,6.3) -- (-19.2,5);
\draw  [orange,thick](-18,5) -- (-19.2,5);

\draw [orange,thick](-18,5) -- (-18.5,7.9);
\draw [orange,thick](-18.5,7.9) -- (-17.2,4.5);

\draw [orange,thick](-18.5,6.3) -- (-16.4,4);
\draw [orange,thick](-18.5,7.1) -- (-16.4,4);
\end{tikzpicture}
\end{minipage}
\begin{minipage}{2.5cm}
\begin{tikzpicture}[scale=0.5]

\draw[rotate around={60:(-17,5)}] (-17,5) arc(0:360: .5cm and 1.3cm);
\draw[rotate around={120:(-20,5)}] (-20,5) arc(0:360: .5cm and 1.3cm);
\draw[rotate around={90:(-18.5,8.5)}](-18.5,8.5) arc(0:360: 1.3cm and 0.5cm);

\draw [fill](-18.5,7.1)circle [radius = 0.1];
\draw [fill](-18.5,6.3)circle [radius = 0.1];

\draw [fill](-17.2,4.5)circle [radius = 0.1];
\draw [fill](-18,5)circle [radius = 0.1];

\draw [fill](-20,4.5)circle [radius = 0.1];
\draw [fill](-19.2,5)circle [radius = 0.1];

\draw [blue,fill](-18.5,7.9)circle [radius = 0.1];
\draw [blue,fill](-16.4,4)circle [radius = 0.1];

\node at(-18,7.8) {$v_1$};

\node at(-16.4,3.6) {$v_2$};

\draw  [orange,thick](-18.5,7.1) -- (-17.2,4.5);
\draw [thick] (-18.5,7.1) -- (-20,4.5);
\draw  [orange,thick](-17.2,4.5) -- (-20,4.5);

\draw [thick](-18.5,6.3) -- (-18,5);
\draw  [thick](-18.5,6.3) -- (-19.2,5);
\draw  [thick](-18,5) -- (-19.2,5);

\draw [thick](-18,5) -- (-18.5,7.9);
\draw [thick] (-18.5,7.9) -- (-17.2,4.5);

\draw  [orange,thick](-20,4.5) -- (-16.4,4);
\draw  [orange,thick](-18.5,7.1) -- (-16.4,4);
\end{tikzpicture}
\end{minipage}
\caption{ Illustration of $G[V_{0}\cup \{v_{1}, v_{2}\}]$.}
\label{l9}
\end{center}
\end{figure}
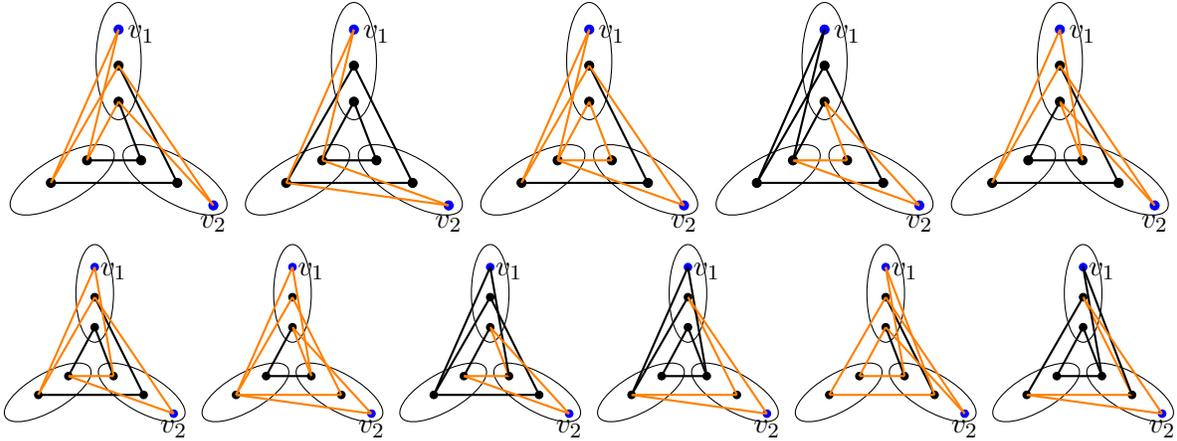
By Claim 5, we have $e(U,V_{0})\ge 2n_{1}+2n_{2}-1-e(G[V_{0}])=2n_{1}+2n_{2}-7$.
Since $|U|=n_{1}+n_{2}+n_{3}-6$ and $e(U,V_{0})\le |U_{0}|+|U|$, we have $|U_{0}|\ge n_{1}+n_{2}-n_{3}-1\ge n_{1}-1>n_{1}-2$. Thus, $U_{0}$ contains at least two vertices $v_{1}$ and $v_{2}$ which come from distinct parts. Then the orange edges in $G[V_{0}\cup \{v_{1}, v_{2}\}]$ (see Figure \ref{l9}) forms one subgraph in $\mathcal{F}$ (see Figure \ref{F}). By Lemma \ref{lemma2}, there exists a rainbow  $C_{4}^{\text{multi}}$, a contradiction.
\end{proof}


\end{document}